\documentclass[12pt,a4paper,reqno]{amsart}

\usepackage{amssymb,amsmath,bm,dsfont,enumerate,manfnt,ragged2e,relsize,setspace,tikz-cd,upgreek,xcolor}

\usepackage{etoolbox}

\usepackage[mathscr]{euscript}

\usepackage{verbatim}

\definecolor{hotmagenta}{rgb}{1.0, 0.11, 0.81}

\makeatletter
\def\SL@inlinetext#1{%
  \SL@interlinetextright{\SL@prlabelname{#1}}%
}
\def\SL@interlinetextleft{\SL@setlefttrue\SL@interlinetext}
\def\SL@interlinetextright{\SL@setleftfalse\SL@interlinetext}
\def\SL@interlinetext#1{%
  \setbox\@tempboxa=\hbox{\showlabelsetlabel{\SL@prlabelname{#1}}}\dp\@tempboxa\z@
  \ifvmode
    \nointerlineskip\vbox to 0pt{
      \hbox to \columnwidth{\box\@tempboxa}}
  \else
    \ifSL@setleft
      \hbox to 0pt{%
        \hss
        \vbox to 0pt{\vss
          \hbox to 0pt{\hss\box\@tempboxa}%
          \showlabelrefline
        }}%
    \else
      \hbox to 0pt{%
        \vbox to 0pt{\vss
          \box\@tempboxa
          \showlabelrefline
        }\hss}%
    \fi
    \penalty10000
  \fi
}
\makeatother

\usepackage[bookmarksnumbered,colorlinks,plainpages,backref,pagebackref,pdfstartview=FitBH]{hyperref}

\usepackage{xparse} 

\NewDocumentCommand{\prelim}{mm}{%
  \operatorname*{#1-{#2}}%
}

\definecolor{refkey}{rgb}{1,.25,.25}

\definecolor{labelkey}{rgb}{0.25,1,.75}

\definecolor{light-gray}{gray}{0.8}

\DeclareMathAlphabet{\matheu}{U}{eur}{m}{n}

\DeclareSymbolFont{euleroperators}{U}{eur}{m}{n}
\SetSymbolFont{euleroperators}{bold}{U}{eur}{b}{n}

\makeatletter
\renewcommand{\operator@font}{\mathgroup\symeuleroperators}
\makeatother

\numberwithin{equation}{section}

\makeatletter 
\def\@seccntformat#1{%
  \protect\textup{\protect\@secnumfont
    \ifnum\pdfstrcmp{subsection}{#1}=0 \bfseries\fi
    \csname the#1\endcsname
    \protect\@secnumpunct
  }%
}
\makeatother

\makeatletter 
\DeclareRobustCommand\bigop[2][1]{%
  \mathop{\vphantom{\sum}\mathpalette\bigop@{{#1}{#2}}}\slimits@
}
\newcommand{\bigop@}[2]{\bigop@@#1#2}
\newcommand{\bigop@@}[3]{%
  \vcenter{%
    \sbox\z@{$#1\sum$}%
    \hbox{\resizebox{\ifx#1\displaystyle#2\fi\dimexpr\ht\z@+\dp\z@}{!}{$\m@th#3$}}%
  }%
}
\makeatother

\makeatletter 

   \newcommand{\adjustedaccent}[1]{%
  \mathchoice{}{}
    {\mbox{\raisebox{-.5ex}[0pt][0pt]{$\scriptstyle#1$}}}
    {\mbox{\raisebox{-.35ex}[0pt][0pt]{$\scriptscriptstyle#1$}}}
                                  }
   
   \newcommand\frownacc[1]{\overset{\adjustedaccent{\bm\frown}}{#1}}
\newcommand{\dol}{\frownacc}
\newcommand{\dbl@overline}[1]{\mathpalette\dbl@@overline{#1}}
\newcommand{\dbl@@overline}[2]{%
  \begingroup
  \sbox\z@{$\m@th#1\overline{#2}$}%
  \ht\z@=\dimexpr\ht\z@-2\dbl@adjust{#1}\relax
  \box\z@
  \ifx#1\scriptstyle\kern-\scriptspace\else
  \ifx#1\scriptscriptstyle\kern-\scriptspace\fi\fi
  \endgroup
}
\newcommand{\dbl@adjust}[1]{%
  \fontdimen8
  \ifx#1\displaystyle\textfont\else
  \ifx#1\textstyle\textfont\else
  \ifx#1\scriptstyle\scriptfont\else
  \scriptscriptfont\fi\fi\fi 3
}
\makeatother

\makeatletter 
\newcommand{\raisemath}[1]{\mathpalette{\raisem@th{#1}}}
\newcommand{\raisem@th}[3]{\raisebox{#1}{$#2#3$}}
\makeatother

\renewcommand{\thefigure}{\arabic{section}.\arabic{figure}}
\renewcommand{\thefigure}{\ifnum\value{section}>0 \arabic{section}.\fi\arabic{figure}}
\renewcommand{\theequation}{\ifnum\value{section}>-1 \arabic{section}.\fi\arabic{equation}}

\newtheorem{thm}[equation]{Theorem}

\newtheorem{cor}[equation]{Corollary}
\newtheorem{lem}[equation]{Lemma}
\newtheorem{prp}[equation]{Proposition}

\theoremstyle{definition}
\newtheorem{dfn}[equation]{Definition}
\newtheorem{ex}[equation]{Example}

\newtheorem{todo}[equation]{\bend\;Todo}

\theoremstyle{remark}
\newtheorem{rem}[equation]{Remark}

\newcommand{\thmref}[1]{Theorem~\ref{#1}}
\newcommand{\prpref}[1]{Proposition~\ref{#1}}
\newcommand{\lemref}[1]{Lemma~\ref{#1}}
\newcommand{\corref}[1]{Corollary~\ref{#1}}
\newcommand{\dfnref}[1]{Definition~\ref{#1}}
\newcommand{\remref}[1]{Remark~\ref{#1}}
\newcommand{\exref}[1]{Example~\ref{#1}}

\newcommand{\secref}[1]{Section~\ref{#1}}

\DeclareMathOperator{\aug}{amp}
\DeclareMathOperator{\bnd}{bnd}
\DeclareMathOperator{\Fun}{Fun}
\DeclareMathOperator{\Meas}{Meas}
\DeclareMathOperator{\mult}{mult}
\DeclareMathOperator{\Perm}{Perm}
\DeclareMathOperator{\Prob}{Prob}
\DeclareMathOperator{\PProb}{PProb}

\DeclareMathOperator{\tail}{tail}

\makeatletter
\patchcmd{\@addmarginpar}{\ifodd\c@page}{\ifodd\c@page\@tempcnta\m@ne}{}{}
\makeatother

\newcommand\note[1]{}

\newcommand\uplabel[1]{\vskip-.7cm\label{#1}\vskip.7cm}

\newcommand{\cond}[1]{
\hskip 4pt
\vbox{\hbox{
\begin{minipage}{130mm}{\noindent\sfemph{#1}}\end{minipage}
}}
\hspace{10000pt minus 1fil}
\medskip
}

\makeatletter
\newcommand*{\tr}{%
  {\mathpalette\@tr{}}%
}
\newcommand*{\@tr}[2]{%
  \raisebox{\depth}{$\m@th#1\intercal$}%
}
\makeatother

\makeatletter
\@namedef{subjclassname@2020}{%
  \textup{2020} Mathematics Subject Classification}
\makeatother

\newcommand\1{{\mathds{1}}}

 \newcommand\Af{{\mathfrak A}}

  \newcommand\As{\mathscr{A}}
 \newcommand\Aut{\operatorname{\sf Aut}}

 \newcommand\Bs{\mathscr{B}}

 \newcommand\Cc{{\mathcal C}}
 
 \newcommand{\cnv}{\DOTSB\bigop[0.9]{\ast}}

\newcommand\da{\downarrow}
\newcommand\de{\delta}

\newcommand\Eb{{\mathbf E}}
\newcommand\Ec{\mathcal{E}}
\newcommand\Ee{\matheu{E}}

\newcommand\ep{\varepsilon}
\newcommand\et{\eta}

 \newcommand\fl{\mathsf{f}}
 
  \newcommand\Fs{\mathscr{F}}

\newcommand\Gc{\mathcal{L}}

\newcommand\ha{\hookrightarrow}
\newcommand\HD{\operatorname{HD}}

 \newcommand{\ka}{\kappa}
 \newcommand{\kb}{\upkappa}
 \newcommand\Kb{{\mathbf K}}
 
 \newcommand\Ks{\mathscr{K}}
 \newcommand\kt{z}

 \newcommand\La{\Lambda}
 \newcommand\la{\lambda}

\newcommand\Mb{{\mathbf M}}

\newcommand\Ms{\mathscr{M}}

\newcommand\Nc{\mathcal{N}}
\newcommand\NN{\mathbb N}

\newcommand\ol{\overline}
\newcommand\Om{\Omega}
\newcommand\om{\omega}

 \newcommand\p{\partial}
\newcommand\Pb{\mathbf P}
\newcommand\Pc{\mathcal P}
\newcommand\Pe{\matheu{P}}
\newcommand\pe{\matheu{p}}
\newcommand\ph{\varphi}
\newcommand\pim{\text{\Large$\bar\pi$}}

\newcommand\qe{{\matheu q}}

 \newcommand\RR{\mathbb R}
 
\newcommand\Rc{\mathcal R}

 \newcommand\sfemph{\emph}
\newcommand\si{\sigma}
 
 \newcommand\supp{\operatorname{supp}}

\newcommand\Tc{\mathcal{T}}
 
\newcommand\Th{\Theta}
\renewcommand\th{\theta}

\newcommand{\toto}{\mathop{\,\sim\joinrel\rightsquigarrow\,}}

\newcommand\upet{\upeta}

\newcommand\we{\matheu w}
 \newcommand\wh{\widehat}

\newcommand{\wlim}{\prelim{\matheu w*}{\lim}}

 \newcommand\wt{\widetilde}

\newcommand\Xb{{\mathbf X}}

\newcommand\Xc{\mathcal{X}}
\newcommand\Xs{\mathscr{X}}

\newcommand\ZZ{{\mathbb Z}}
 
\newcommand\Zp{{\ZZ_+}}
\newcommand\ZX{\ZZ[\Xs]}
\newcommand\ZpX{\ZZ_+[\Xs]}

\begin{document}

\title{Limit distributions of branching Markov chains}

\author{Vadim A. Kaimanovich}
\address{\parbox{1.0\linewidth}
{
Department of Mathematics and Statistics, University of Ottawa, 150 Louis
Pasteur, Ottawa ON, K1N 6N5, Canada\\
{\tt vkaimano@uottawa.ca, vadim.kaimanovich@gmail.com}}}
\author{Wolfgang Woess}
\address{\parbox{1.0\linewidth}
{
Institut f\"ur Diskrete Mathematik,
Technische Universit\"at Graz,
Steyrergasse 30, 8010~Graz, Austria\\
{\tt woess@tugraz.at}}}
\begin{abstract}
We study branching Markov chains on a countable state space (space of types) $\Xs$ 
with the focus on the qualitative aspects of the limit behaviour of the evolving
empirical population distributions.
No conditions are imposed on the multitype offspring distributions at the points of $\Xs$
other than to have the same average and to satisfy a uniform
$L \log L$ moment condition. We show that the arising population martingale is uniformly
integrable. Convergence of population averages of the branching chain is then put in 
connection with stationary spaces of
the associated ordinary Markov chain on $\Xs$ (assumed to be irreducible and transient).
This is applied, in particular, to the
boundaries of appropriate compactifications of $\Xs$.
Final considerations consider the general interplay between the measure theoretic boundaries of
the branching chain and the associated ordinary chain.
\end{abstract}

\date{May 2022}

\thanks{The second author was supported by Austrian Science Fund project FWF: P31889-N35 during a visit at University of Ottawa in 2019.} 

\maketitle

\thispagestyle{empty}

\setcounter{section}{-1}
\section{Introduction}

There is a large body of literature devoted to the quantitative aspects of branching random
walks on the additive group of real numbers and to the behaviour of the associated martingales,
e.g., see {\sc Shi} \cite{Shi15} and the references therein. In what concerns more general
state spaces rich enough to have a non-trivial topological boundary at infinity (like,
for instance, infinite trees), it is natural to ask about the limit behaviour of the
branching populations in geometric terms. Non-trivial limit sets of random population
sequences were first exhibited by {\sc Liggett}~\cite{Liggett96} for branching random
walks on regular trees. This was pursued further for branching random walks on free groups  by
\textsc{Hueter~-- Lalley} \cite{Hueter-Lalley00}, on more general free products by
\textsc{Candellero -- Gilch -- M\"uller}  \cite{Candellero-Gilch-Muller12}, and very recently to
random walks on hyperbolic groups by
\textsc{Sidoravicius -- Wang --Xiang}~\cite{Sidoravicius-Wang-Xiang20p}. See also
\textsc{Benjamini -- M\"uller} \cite[Section 4.1]{Benjamini-Muller12} for a number
of conjectures concerning the trace and limit sets of branching random walks. Some
answers were given by
\textsc{Candellero -- Roberts}~\cite{Candellero-Roberts15} and \textsc{Hutchcroft}~\cite{Hutchcroft20}.

We are looking at branching random walks from a different and apparently novel angle.
We are interested in the random limit boundary measures arising from the empirical
distributions of sample populations. Unlike with the limit sets, the very existence
of the limit measures is already a non-trivial problem. We consider and solve this
problem in two different setups: in the topological one (when the boundary of the
state space is provided by a certain compactification) and in the measure-theoretical
one (when we are dealing with the Poisson or exit boundary of the underlying Markov
chain on the state space).
\pagebreak[4]

\begin{dfn} \label{dfn:bmc}
Let $\Xs$ be a countable set, the \sfemph{state space.}

\smallskip

(a)
A \sfemph{population} on $\Xs$ is a finitely supported function $m: \Xs \to \ZZ^+$, also viewed as
a multiset, so that $x \in m$ means that $m(x) > 0$ and $m(x)$ is the number of particles (members of
the population) situated at $x \in \Xs$. Thus, the same location can be shared by several particles.
We emphasise the difference between a population $m\in\Ms=\ZpX$ (which is a \emph{multiset}) and its \sfemph{support} $\supp m = \{x\in\Xs: m(x)>0 \}$ which is a plain 
subset of $\Xs$.

\smallskip

(b)
A \sfemph{branching Markov chain} is a time homogeneous Markov chain on the space of populations
$\Ms$ whose transitions $m\toto m'$ are determined by a family of
\sfemph{branching probability distributions}
$\Pi_x\,,\; x\in\Xs,$ in the following way: each particle of the population $m$ is replaced with a population independently sampled from the distribution $\Pi_x$ determined by the position $x$ of the particle; the result of this procedure is the population~$m'$.
That is, for $y \in \Xs$, $m'(y)$ is the sum of the random offspring numbers which each $x \in m$
places at position $y$.
\end{dfn}

The elements of the state space $\Xs$ are often referred to as \emph{types}, and then one
talks about \emph{multi-type branching processes} (rather than branching Markov chains).
They were first
considered by \textsc{Kolmogorov}~\cite{Kolmogoroff41}. The explicit general definition was given by
\textsc{Harris}~\cite[Section~III.6]{Harris63}. There is an ample literature on multi-type branching
processes, in discrete as well as continuous time. The reader is referred to the survey by
\textsc{Ney}~\cite{Ney91} for a historical account and
for general information on this field. More relevant literature will be outlined further below.

The reason for our choice of terminology is that we
take a geometrical point of view and have in mind a spatial structure of~$\Xs$.
Our branching Markov chain is a sequence $\Mb = (M_n)_{n \ge 0}$ of random populations, and we are interested
in its evolution in that space. In particular, we are interested in the behaviour of the
sequence of \sfemph{empirical distributions}
\begin{equation}\label{eq:empdis}
 \dol M_n = \frac{1}{\|M_n\|} M_n\,,\quad \text{where} \quad \|m\| = \sum_x m(x) \; \text{ for }\; m \in \Ms\,.
\end{equation}
Associated with $\Pi_x$ there is the \sfemph{offspring distribution\footnotemark
\footnotetext{\;We distinguish \emph{branching distributions} (measures on the space of populations $\Ms$) and \emph{offspring distributions} (measures on $\Zp$).}
$\pi_x$ at} $x$, where
\begin{equation}\label{eq:offspring}
\pi_x(k) = \Pi_x ( \{ m \in \Ms : \|m\| = k\} )\,.
\end{equation}
The \sfemph{branching ratio} $\ol \pi_x$ at~$x$ is the first moment of $\pi_x$
(the expected offspring number at~$x$), and
$\ol\pi_{x,y}$ denotes the expected offspring number which is placed at $y \in \Xs$ under the
distribution $\Pi_x\,$.

A branching Markov chain gives rise to the \sfemph{underlying} (also caled \sfemph{base})
``ordinary'' \sfemph{Markov chain} on the state space $\Xs$ whose transition operator (matrix)
$P$ is given by
\begin{equation}\label{eq:transprob}
 p(x,y) = p_x(y) = \ol\pi_{x,y} / \ol\pi_x\,.
\end{equation}
We write $p^{(n)}(x,y)$ for its $n$-step transition probabilities and
$G(x,y) = \sum_{n=0}^{\infty} p^{(n)}(x,y)$ for the associated Green function.

Our basic assumptions, beginning with Section \ref{sec:pm},
are the following.
\begin{equation} \label{ass:ne} \tag{\sfemph{NE}}
\cond{The population cannot die out and has non-trivial branching, that is, $\pi_x(0)=0$ and
$\pi_x(1) < 1$ for all $x \in \Xs$.}
\end{equation}
\begin{equation} \label{ass:br} \tag{\sfemph{BR}}
\cond{The branching ratio is constant and finite, i.e., there is $\rho<\infty$ such that $\ol\pi_x=\rho$ for all points $x\in\Xs$.}
\end{equation}
\begin{equation} \label{ass:tc} \tag{\sfemph{TC}}
\cond{The underlying Markov chain is transient, and all its states communicate, i.e., $0 < G(x,y) < \infty$ for all $x,y\in\Xs$.}
\end{equation}
Note that then $\rho > 1$.
The assumptions can be relaxed, but they simplify some technicalities whithout
compromising the conceptual spirit.

In Section \ref{sec:general}, we set up a rigorous framework and present general classes of
examples, including a discussion and many references.

One of our main aims is to study the boundary behaviour of the sequence \eqref{eq:empdis}
of empirical distributions, or very similarly, of the sequence
\begin{equation}\label{eq:WW_n}
\frac{1}{\rho^n} M_n \in \Meas(\Xs)
\end{equation}
under Assumption \eqref{ass:br}. Assuming that the state space $\Xs$ is endowed with a suitable
\emph{compactification} $\ol\Xs = \Xs \cup \p\Xs$, there has been a body of interesting
work considering the \emph{limit set} of the \emph{trace} (set of visited points) of the branching Markov chain on the boundary $\p\Xs$. For more details and references, see \S \ref{subsec:lim}.
Our goal is to shift the focus and to look, instead of the \emph{limit sets}, at the random \emph{limit boundary measures} obtained as the weak* limits of the empirical distributions \eqref{eq:empdis}, resp., the measures \eqref{eq:WW_n}.

Before we embark on this study, a comparison of \eqref{eq:empdis} and \eqref{eq:WW_n} reveals
that we need to understand the behaviour of the following sequence.

\begin{dfn} \label{dfn:pop}
The \sfemph{population martingale} of a branching Markov chain satisfying \eqref{ass:br}
is the sequence of random variables (functions on the path space)
\begin{equation} \label{eq:pm}
W_n = W_n(\Mb) = \frac{1}{\rho^n} \, \|M_n\|\,.
\end{equation}
We call its a.s.\ pointwise limit
$$
W_\infty = W_\infty(\Mb) = \lim_n W_n(\Mb) \;.
$$
the \sfemph{limit population ratio}.
\end{dfn}
We emphasise that even though the branching ratio is assumed to be constant, the offspring
distributions $\pi_x$ themselves are allowed to be different. What we need here is an
extension of the classical theorem of \textsc{Kesten -- Stigum} \cite{Kesten-Stigum66}
which says that for a single offspring distribution $\pi$, the population martingale is
uniformly integrable
if and only if $\pi$ satisfies the \sfemph{$L \log L$ moment condition.}

This issue is dealt with in Section \ref{sec:pm}.  We introduce the \sfemph{uniform} $L \log L$ moment condition for the family $(\pi_x)_{x \in \Xs}$. Under this condition, we prove that the population
martingale is uniformly integrable (Theorem \ref{thm:ll}) and that $W_\infty$ is almost surely
strictly positive for any initial population (Theorem \ref{thm:pos}).

Section \ref{sec:tc} is the central one: in order to study boundary convergence of
the sequence of empirical distributions \eqref{eq:empdis}, we first consider \emph{stationary
spaces} for the underlying Markov chain. These are measurable spaces equipped with a \emph{$P$-harmonic
system} of probability measures $\ka_x\,$, $x \in \Xs$, see Definition \ref{dfn:stat}. We can then
study the sequence of random measures
$$
\ka_{M_n} = \sum_{x \in M_n} \ka_x \,.
$$
A particularly interesting case is the one where we have
a compactification $\ol \Xs$ of the state space with separable boundary
$\partial \Xs = \ol\Xs \setminus \Xs$. What we want is its compatibility with the underlying Markov chain
$\Xb = (X_0\,,X_1\,,\dots)$ in the sense that 
$X_n$ converges almost surely to a $\partial \Xs$-valued random variable $X_{\infty}$ for any
starting point $x \in \Xs\,$. Endowed with the associated limit distributions $\ka_x\,$, the boundary is a stationary space. Our main Theorem \ref{thm:kconv} states almost sure weak* convergence of the normalised random measures $\frac{1}{\rho^n}\, \ka_{M_n}$
to a positive Borel measure $\kb_{\Mb}$ under the uniform $L \log L$ moment
condition.
Further, assume that the compactification is \emph{Dirichlet regular,} which means that every continuous function on $\partial \Xs$ has a continuous
continuation to $\ol\Xs$ which is $P$-harmonic in $\Xs$. In this situation,
the random measures $\frac{1}{\rho^n}\, M_n$ themselves converge (weak*) to
$\kb_{\Mb}$ almost surely. As a consequence, we obtain in Theorem \ref{thm:emp} that
the random probability measures
$\frac{1}{\|M_n\|}\, \ka_{M_n}$ also converge almost surely. In particular, in the
case of Dirichlet regularity, the sequence of empirical distributions \eqref{eq:empdis} converges
almost surely to a random probability measure on the boundary -- our primary goal.

In the last parts of Section  \ref{sec:tc}, we review geometric, resp. algebraic adaptedness conditions
of the transition probabilities of the underlying chain $(X_n)$ to a given graph or group structure
of the
state space, in which case we speak of a \emph{random walk.} Then we recall a few typical
compactifications and explain how theorems \ref{thm:kconv} and \ref{thm:emp} apply.

In Section \ref{sec:corres}, we shift our attention from topological to measure theoretic
boundary theory. We start by explaining in some detail the Poisson boundary
for a general (i.e., not necessarily group invariant) Markov chain on a countable state space
and its relation with
the tail boundary. We elucidate the relationship between compact stationary spaces and quotients of the
Poisson boundary. Our goal in this section is to establish a link, in the measure-theoretic
context, between the boundaries of a branching Markov chain and that of
the underlying chain. Theorem \ref{thm:bdry} provides a natural \emph{transfer operator}
from the tail boundary of the latter to that of the former, which is Markov on the respective
Banach spaces of essentially bounded functions. Theorem \ref{thm:bb} clarifies how this operator
descends to
a compact stationary space for the base chain. The final Theorem \ref{thm:bmeas} explains the
importance of the above Markov transfer operator: in the topological
context of \S \ref{sec:tc}, its range is precisely the set of random limits
of the sequence of empirical distributions.

\smallskip

\textbf{Acknowledgement.} The beginning of this work has its roots in a discussion of the authors
with Elisabetta Candellero in Warwick in 2015, where the first author of the present paper
proposed to study the behaviour of the sequence of empirical distributions rather than of individual
genealogical lines.

\section{Basic notions} \label{sec:general}

\subsection{General framework} \label{subsec:frame}

Here, we set up a general rigorous framework for our main objects.
Given our countable state space $\Xs$, we use the notations below for the following spaces.
{\setlength{\leftmargini}{23pt}
\begin{itemize}
\item
$\Fun(\Xs)$ -- bounded real-valued functions on $\Xs$;
\item
$\Meas(\Xs)$ -- non-negative (not necessarily finite) measures on $\Xs$;
\item
$\Prob(\Xs)\subset\Meas(\Xs)$ -- probability measures on $\Xs$;
\item
$\Ms= \ZpX \subset \ZZ^\Xs_+$ -- finitely supported $\Zp$-valued functions,  \emph{populations}
as in Def.~\ref{dfn:bmc}, finite \emph{multisets} on $\Xs$.
\end{itemize}}
When talking about integration we use the ``pairing notation'' $\langle \mu, f \rangle$ to denote the integral of a function $f$ with respect to a measure $\mu$ (a sum in the discrete case).

One can also treat $\Ms$ as a subspace of $\Meas(\Xs)$ that comprises all finite
non-negative integer valued measures on $\Xs$ (sometimes called \emph{occupation measures}). Thus, $\|m\|$
is the total mass of $m\in\Ms$.
If $\Xs$ is a group, then $\Ms=\ZpX$ is precisely the set of all non-negative elements of the
\emph{group algebra} $\ZX$ of $\Xs$ over $\ZZ$ (which is the reason for our notation).
In this situation, the map that assigns to any population its amplitude (size) 
\begin{equation}
\label{eq:aug}
\aug: \Ms=\ZpX \to \Zp \;, \qquad m \mapsto \|m\| \;,
\end{equation}
is nothing but a restriction of the corresponding \emph{augmentation homomorphism}. 
We use this
term for the additive \sfemph{augmentation map} $\aug$ in our more general setup
as well. Applied to a measure $\Pi\in\Meas(\Ms)$, it gives rise to the image measure
$$
\pi = \aug(\Pi) \in \Meas(\Zp),
$$
given as in \eqref{eq:offspring} (without the $x$ in the index). The \sfemph{barycentre} (the \sfemph{first moment}) of $\pi$ is denoted by
\begin{equation} \label{eq:fm}
\ol\pi = \sum_k \pi(k) \cdot k \;.
\end{equation}
If $\Pi$ (and therefore $\pi$ as well) is a probability measure, then $\pi$ is the 
\sfemph{size distribution} of the populations sampled from $\Pi$, and $\ol\pi$ is their 
\sfemph{average size}.
We use the same notation
\begin{equation} \label{eq:ol}
\Pi\mapsto \ol\Pi = \sum_{m\in\Ms} \Pi(m) \cdot m \;,\qquad \Meas(\Ms)\to\Meas(\Xs) \;,
\end{equation}
for the \sfemph{barycentre map} on the space $\Meas(\Ms)$ obtained by linear extension of the
mapping
$$
\de_m \mapsto \ol{\de_m} = m \;, \qquad m\in \Ms \;.
$$
The horizontal arrows in the following commutative diagram represent the barycentre maps from $\Meas(\Ms)$ and $\Meas(\Zp)$ to $\Meas(\Xs)$ and $\RR$, respectively,
\begin{equation}
\begin{tikzcd}
\Meas(\Ms) \arrow[r] \arrow[d,"\aug"] & \Meas(\Xs) \arrow[d,"\|\cdot\|"] && \Pi \arrow[r,mapsto] \arrow[d,mapsto] & \ol\Pi \arrow[d,mapsto] &\\
\Meas(\Zp) \arrow[r] & \;\; \RR_+ & \hskip -1.5cm, & \pi \arrow[r,mapsto]  & \ol\pi &
\end{tikzcd}
\end{equation}
In particular, the total mass of the barycentre measure $\ol\Pi$ is
$$
\left\| \ol\Pi \right\| = \ol{\aug(\Pi)} = \ol\pi \;,
$$
so that if $\Pi$ is a probability measure, then
$$
\ol\pi = \sum_{m \in \Ms} \Pi(m)\,\|m\|  
$$
is precisely the average size of the populations sampled from $\Pi$.
If $\ol\pi=\left\| \ol\Pi \right\|<\infty$ (for instance, if $\Pi$ is finitely supported), then the normalisation of $\ol\Pi$ produces the \sfemph{displacement distribution}
\begin{equation} \label{eq:disp} 
p = \frac{\ol\Pi}{\ol\pi} \in \Prob(\Xs), \quad \text{i.e.,} \quad
p(y) = \frac{1}{\ol\pi} \sum_{m \in \Ms} \Pi(m)\,m(y) \; \text{ for }\; y \in \Xs.
\end{equation}

Since the population space $\Ms$ is contained in the commutative group $\ZX$, which is an additive semigroup, one can define in the usual way the \sfemph{convolution} of two measures on $\Ms\,$:
\begin{equation} \label{eq:conv}
\Pi*\Pi'(m'')= \sum_{m+m' = m''} \Pi(m)\,\Pi'(m')\,.
\end{equation}
If both arguments are probability measures, then $\Pi*\Pi'$ is the distribution of the sum $M+M'$,
where the random summands are independently sampled from the respective distributions $\Pi$ and $\Pi'$. Clearly,
\begin{equation} \label{eq:add}
\ol{\Pi*\Pi'} = \ol\Pi + \ol{\Pi'} \qquad \forall\,\Pi,\Pi'\in\Prob(\Ms) \;,
\end{equation}
in particular,
$$
\left\| \ol{\Pi*\Pi'}  \right\| = \left\| \ol\Pi \right\| + \left\| \ol{\Pi'} \right\| \;.
$$

\subsection{Implementation for branching Markov chains} \label{subsec:bd}

For branching Markov chains,
we use the notation of \S \ref{subsec:frame}  for the various objects associated with the
branching distributions $\Pi_x$ by adding the subscript $x$. As anticipated in the Introduction, $\pi_x$ is the offspring distribution at $x$.
The displacement distribution \eqref{eq:disp} associated with the probability measure $\Pi_x$ is the transition kernel $p_x$ of \eqref{eq:transprob}.
As follows from \dfnref{dfn:bmc} and the definition of the convolution operation \eqref{eq:conv}, the transition probabilities of the branching Markov chain $\Mb=(M_0, M_1, \dots)$ are the convolutions
\begin{equation} \label{eq:PiM}
\Pi_m = \cnv_{x\in m} \Pi_x \;, \qquad m\in\Ms \;,
\end{equation}
where we treat the populations $m$ as multisets, so that each point from the support of $m$ is taken with its multiplicity. That is, the probability of the move $m \toto m'$ is $\Pi_m(m')$.
We denote by $\pmb\Pe_{\!\Th}$ the probability measure on the space~$\Ms^\Zp$ of \sfemph{sample paths} of
$\Mb$ corresponding to the initial distribution $\Th\in\Prob(\Ms)$.
We use the notation $\pmb\Pe_m=\pmb\Pe_{\de_m}$ for the initial distribution $\Th=\de_m$ concentrated at a single population $m\in\Ms$, and $\pmb\Pe_{\!x}=\pmb\Pe_{\de_x}$ if $m=\de_x$ is the singleton at a point $x\in\Xs$. The respective expectations are denoted by $\pmb\Ee_\Th\,, \pmb\Ee_m\,, \pmb\Ee_x\,$. All these measures on the path space are absolutely continuous with respect to the common \sfemph{initial full support} class of the measures $\pmb\Pe_{\!\Th}$ corresponding to the initial distributions $\Th$ with $\supp\Th=\Ms$.
\begin{equation} \label{fs} \tag{\sfemph{FS}}
\cond{It is to the initial full support measure class that we refer when we use the expression ``almost everywhere'' without specifying a measure on the path space.}
\end{equation}

The \sfemph{transition operator} of the branching random walk is
\begin{equation} \label{eq:pe}
\Pc F (m) = \langle \Pi_m, F \rangle = \pmb\Ee_m\, F(M_1)
\end{equation}
It is well-defined not only on the space $\Fun(\Ms)$ of bounded functions on $\Ms$, but also
for all non-negative positive functions (allowed to take the value $+\infty$). Following the
standard probabilistic convention we use the postfix notation
$$
\Th \Pc = \sum_m \Th(m)\, \Pi_m
$$
for the action of the dual operator on the space $\Meas(\Ms)$ of positive measures on $\Ms$, so that $\Th\Pc$ is the time 1 marginal distribution of the measure $\pmb\Pe_{\!\Th}$.

For the underlying Markov chain $\Xb = (X_0,X_1,\dots)$ with transition probabilities given by \eqref{eq:transprob}, resp. \eqref{eq:disp}, we denote in the same way as above
the measures on the space $\Xs^\Zp$ of sample paths by $\Pb_{\!\th}$ (or $\Pb_{\!x}=\Pb_{\!\de_x}$, if
the initial distribution $\th$ is concentrated at a single point $x\in\Xs$), the respective
expectations by $\Eb_\th,\Eb_x$, and the transition operator by
\begin{equation} \label{eq:p}
Pf(x) = \langle p_x, f \rangle = \Eb_x f(X_1) \;.
\end{equation}

\subsection{The lifting operator} \label{subsec:trans}

For a function $f$ on $\Xs$, we denote by
\begin{equation} \label{eq:ext}
\wt f (m) = \langle m, f \rangle = \sum_{x\in\Xs} m(x) f(x)
\end{equation}
its \sfemph{lift} to the space of populations $\Ms$. In particular,
\begin{equation} \label{eq:1}
\wt\1(m) = \|m\| \qquad\forall\,m\in\Ms
\end{equation}
for the function $\1(x)\equiv 1$ on $\Xs$. The \sfemph{lifting operator}
\begin{equation} \label{eq:L}
f \mapsto \wt f = L f \;, \qquad \Fun(\Xs)\to\Fun(\Ms) \;,
\end{equation}
is dual to the barycentre map \eqref{eq:ol}, i.e.,
$$
\bigl\langle \, \ol\Th, f \bigr\rangle = \bigl\langle \Th, \wt f \;\bigr\rangle \qquad \forall\,\Th\in\Meas(\Ms), \; f\in\Fun(\Xs) \;.
$$
Therefore, the barycentre map can be written in the postfix notation as
$$
\Th \mapsto \ol \Th = \Th L \;, \qquad \Meas(\Ms)\to\Meas(\Xs) \;,
$$
\begin{prp} \label{prp:tr}
The transition operators $\Pc$ \eqref{eq:pe} and $P$ \eqref{eq:p} of the branching Markov 
chain and of the underlying chain, respectively, satisfy the commutation relation
\begin{equation} \label{eq:comm}
\Pc L = L \pim P \;,
\end{equation}
where $\pim$ denotes the operator of multiplication by the branching ratio function $\ol\pi:x\mapsto\ol\pi_x$ (see subsection \ref{subsec:frame}). In other words,
$$
\Pc\wt f = \wt {\ol\pi\!\cdot\! P f} \qquad\forall\,f\in\Fun(\Xs) \;,
$$
and
$$
\ol{\Th\Pc} = \left( \ol\pi\cdot\ol\Th\, \right) P \qquad\forall\,\Th\in\Meas(\Ms) \;.
$$
\end{prp}

\begin{proof}
It is more convenient to prove the commutation relation for the dual operators acting on measures. By linearity it is enough to consider the situation when $\Th=\de_m$ is the delta measure at a population $m\in\Ms$:
$$
\begin{aligned}
\ol{\de_m\Pc}
&= \sum_x m(x) \ol{\,\Pi_x}
= \sum_x m(x) \ol\pi_x \,p_x \\
&= (\ol\pi\cdot m) P
= \left( \ol\pi\cdot\ol{\de_m}\, \right) P \;.
\end{aligned}
$$
\end{proof}

We recall that a function $f$ is called \sfemph{harmonic} with respect to a transition operator~$P$ if $Pf=f$, and it is called \sfemph{$\la$-harmonic} for an eigenvalue $\la\in\RR$ if $Pf=\la f$.

\begin{cor} \label{cor:har}
If the branching ratio $\ol\pi_x\equiv\rho$ is constant, then for any $\la$-harmonic function $f$ of the underlying chain its lift to the population space $\Ms$ is $\la\rho$-harmonic for the branching Markov chain.
\end{cor}

This property will play a key role in the rest of the paper.

\subsection{Examples of branching Markov chains} \label{subsec:exa}

\begin{ex} \label{ex:gw}
If all branching distributions $\Pi_x$ are concentrated on one-point configurations (i.e., all offspring distributions $\pi_x$ are just $\de_1$, and $\ol\pi_x\equiv 1$), then the barycentres~$\ol\Pi_x$ are probability measures, so that in this situation the branching Markov chain consists in running \emph{independent sample paths of the underlying Markov chain} issued from each particle of the initial population.
\end{ex}

\begin{ex} \label{ex:GW}
If the state space $\Xs$ is a singleton, then the size is the only parameter that describes populations
on $\Xs$, and a branching Markov chain over $\Xs$ is determined just by a single offspring
distribution $\Pi\cong\pi=\aug(\Pi)$ on $\Ms\cong\Zp$. Therefore, it is nothing but the usual \emph{Galton~-- Watson branching process} determined by $\pi$.
\end{ex}

\begin{ex} \label{ex:bd}
For a probability measure $\mu$ on $\Xs$, we denote by $\mu^k\in\Prob(\Ms)$ the image of the product measure $\mu^{\otimes k}$ on $\Xs^k$ under the map
$$
\Xs^k\to \Ms\,, \qquad (x_1,\dots,x_k)\mapsto \de_{x_1}+\dots+\de_{x_k} \,.
$$
Given a distribution $\pi\in\Prob(\Zp)$ and a Markov chain on $\Xc$ with the transition
probabilities $p_x\in\Prob(\Xc)$, the family of branching distributions
\begin{equation} \label{eq:bd}
\Pi_x = \sum_{k\ge 0} \pi(k) \cdot p_x^{\otimes k}
\end{equation}
determines then a branching Markov chain with \emph{independent branching and displacement}, for which
all offspring distributions $\pi_x$ coincide with $\pi$, and the displacement distributions 
are $p_x\,$. Any particle occupying a position $x\in\Xc$ fissions into a $\pi$-distributed 
random number of new particles, 
and each new particle moves to a new $p_x$-distributed position independently of all 
other particles. Branching and displacement can be fully decoupled by first generating a 
random Galton -- Watson tree $T$ with the offspring distribution $\pi$, and then running 
the \emph{$T$-indexed Markov chain} with the transition probabilities $p_x$ 
(e.g., see \textsc{Aldous} \cite[Section 6, p.\ 64]{Aldous91a} and 
\textsc{Benjamini~-- Peres} \cite{Benjamini-Peres94}).
\end{ex}

\begin{ex} \label{ex:bs}
In case all offspring distributions $\pi_x$ of a branching Markov chain coincide with a
common distribution $\pi$, this does in no way imply that offspring and displacement are
independent. It just means that all branching distributions $\Pi_x$ have the form
\note{W: The formal definition \eqref{eq:mult} with $\mult$ took me quite a while to be deciphered
from that formalism to its meaning. Many readers would be subject to the same effect. Since
$\mult$ came up only here and only twice, I am convinced that it is not worth while to introduce that
formalism, while writing it out as I
did will help the understanding without consuming additional space.}
$$
\Pi_x = \sum_{k\ge 0} \pi(k) \cdot \Pi_x^k \;,
$$
where $\Pi_x^k$ are probability measures on size $k$ populations. For instance, if we let
$m_y^k = k\cdot \de_y$ be the population with $k$ particles at $y$ and none elsewhere, then we
can consider the branching distributions
\begin{equation} \label{eq:mult}
\Pi_x = \sum_{k\ge 0} \pi(k) \sum_{y \in \Xs} p_x(y) \cdot \delta_{m_y^k}\,.
\end{equation}
The corresponding transition distributions are again $p_x\,$. In the branching Markov chains determined
by both \eqref{eq:bd} and \eqref{eq:mult} first one samples a Galton -- Watson tree with the offspring distribution $\pi$ and then equips this tree with the transitions sampled from the appropriate transition distributions $p_x$. However, in \exref{ex:bd} the independently sampled transitions are parameterised by the \emph{edges} of the tree, whereas for the chain determined by \eqref{eq:mult} they are parameterised by the \emph{vertices} of the tree (so that the transition is the same for all edges issued from the same vertex in the direction away from the root). It might be interesting to look at the branching Markov chains determined by convex combinations of the measures \eqref{eq:bd} and \eqref{eq:mult}.
\end{ex}

\begin{ex} \label{ex:gm}
One can also consider a more general situation than in \exref{ex:bd} with the offspring distributions $\pi_x$ being \emph{space dependent}, 
although the displacement is still governed by the transition probabilities $p_x$ of an underlying Markov chain on $\Xs$ (e.g., see \textsc{Menshikov -- Volkov} \cite{Menshikov-Volkov97} and \textsc{Gantert -- M\"uller} \cite{Gantert-Muller06}). In this case the resulting branching Markov chain is determined by the branching distributions
$$
\Pi_x = \sum_{k\ge 0} \pi_x(k) \cdot p_x^{\otimes k} \;.
$$
In the context of this example, 
the dependence of $\pi_x$ on $x$ is often referred to as an \emph{environment}; if it is random, then one talks about \emph{branching Markov chains in random environment}, see \textsc{Comets~-- Menshikov~-- Popov} \cite{Comets-Menshikov-Popov98}. The term ``environment'' is also used to describe the generalisation of the Galton -- Watson process that allows the offspring distribution to depend on the generation number,  see e.g. \textsc{Athreya -- Ney} \cite[Section~VI.5]{Athreya-Ney72}. By passing to the space-time process (see \secref{sec:tp}) the latter model becomes a particular case of the former one.
\end{ex}

\begin{ex} \label{ex:gr}
If $\Xs$ is a group, then it makes sense to consider the assignments $x\mapsto \Pi_x$ \emph{equivariant} with respect to the natural action of $\Xs$ on the population space $\Ms=\Zp[\Xs]$ by translations, i.e., such that all branching distributions $\Pi_x$ are the translates of a \emph{single} probability measure $\Pi\in\Prob(\Ms)$ (the branching distribution at the group identity). By analogy with ordinary \emph{random walks on groups}, we then talk about \sfemph{branching random walks}. In particular, in this case the offspring distributions $\pi_x$ all coincide with the size distribution~$\pi$ of the measure~$\Pi$, the branching ratios (offspring averages) $\ol\pi_x$ all coincide with $\ol\pi$, and the transition probabilities~$p_x$ are the translates of the displacement distribution~$\mu=\ol\Pi/\ol\pi$, the \sfemph{law} of the random walk on the group: $p_x(y) = \mu(x^{-1}y)$. In the same vein one can also consider the situation when $\Xs$ is endowed with a group action (transitive, quasi-transitive, or a more general one), and the map $x\mapsto\Pi_x$ is equivariant with respect to this action (cf.\ \textsc{Kaimanovich -- Woess} \cite{Kaimanovich-Woess02} and Subsection \ref{subsec:adapt}).
\end{ex}

\subsection{Limit sets vs.\ limit measures} \label{subsec:lim}

Before plunging into \emph{medias res} we outline the earlier approach to the boundary behaviour of branching Markov chains which served as our motivation. For the branching Markov chain
$\Mb=(M_n)$, its \sfemph{trace}
\note{V: Is the property that the trace is almost surely the whole state space the same as the strong recurrence in the sense of M\"uller \cite{Muller08} or Benjamini - M\"uller \cite{Benjamini-Muller12}? In both papers only the chains with independent branching and displacement are considered, but it shouldn't make much difference.}
$$
\supp\Mb = \bigcup_{n\ge 0} \supp M_n \subset \Xs
$$
is the random set of all points from the state space which are charged (or visited) by at least one of the populations $M_n$. Assuming that the state space $\Xs$ is endowed with a \emph{compactification}
$\ol\Xs = \Xs \cup \p\Xs$ (see \secref{sec:zoo} below for definitions and examples), one can then define, in the usual way, the \sfemph{limit set} of a sample path as the boundary of its trace with respect to this compactification:
$$
\La(\Mb) = \ol{\supp\Mb} \setminus \supp\Mb = \ol{\supp\Mb} \cap \p\Xs \;.
$$
Notions of of \sfemph{recurrence and transience}
for branching Markov chains with independent branching and displacement have been studied by \textsc{Benjamini~-- Peres} \cite{Benjamini-Peres94}, \textsc{M\"uller} \cite{Muller08}, \textsc{Bertacchi -- Zucca}~\cite{Bertacchi-Zucca08} (in continuous time) and \textsc{Benjamini -- M\"uller} \cite{Benjamini-Muller12}; see also
{\sc Woess}~\cite[\S 5.C]{Woess09} for a simplified approach.
If the branching Markov chain is recurrent in the sense that $\supp\Mb=\Xs$ for almost
all sample paths, then obviously the limit set $\La(\Mb)$ coincides almost surely with the whole boundary
$\Xs$. Otherwise, proper traces $\supp\Mb\neq\Xs$ may lead to proper limit sets $\La(\Mb)\neq\p\Xs$, and
it makes sense to look at their properties.

Regarding non-trivial limit sets, see the referenes given at the beginning of the introduction.
Note that in those papers only branching random walks with independent branching
and displacement (as described in \exref{ex:bd}) were considered.

Free groups have served as the ``true touchstone'' in the non-commutative random walk theory for the last 60 years, so let us describe the situation with them in more detail (for instance, see \textsc{Ledrappier} \cite{Ledrappier01} and the references therein for more background). Let~$\As$ be a finite \emph{alphabet} of cardinality $d\ge 2$, and $\Fs$ be the \emph{free group} of rank $d$ generated by~$\As$. We fix a symmetric probability measure $\mu$ with support $\As\cup\As^{-1}$; the simplest case is when $\mu$ is equidistributed on $\As\cup\As^{-1}$, so that the random walk
$(\Fs,\mu)$ is just the simple random walk on the homogeneous Cayley tree of the free group. Further, let $\pi$ be the geometric distribution on $\NN$ with parameter $p\in(0,1)$ and mean $\rho=1/p$. We can now consider the branching random walk with independent branching and displacement determined by the underlying random walk $(\Fs,\mu)$ and the offspring distribution $\pi$.

If $\rho r>1$, where $r=r(\Fs,\mu)$ is the spectral radius \eqref{eq:rr} of the random walk
$(\Fs,\mu)$, then almost surely $\supp\Mb=\Fs$ and $\La(\Mb)=\p\Fs$. In the case $\rho r \le 1$, Hueter and Lalley, extending the above cited result by Liggett related to simple random walk, proved that the Hausdorff dimension $\HD\La(\Mb)$ of the limit set with respect to a natural metric on $\p\Fs$ is almost surely constant and obtained an explicit formula for it \cite[Theorem 1]{Hueter-Lalley00}. In particular, it satisfies the inequality
$$
\HD \La(\Mb) \le \frac12 \HD \p\Fs \;,
$$
and $\HD\La(\Mb)\to 0$ as $\rho \to 1$ from above. This result was extended to branching random walks on free products of finitely generated groups under less restrictive conditions by \textsc{Candellero~-- Gilch -- M\"uller} \cite[Theorems 3.5 and 3.10]{Candellero-Gilch-Muller12}, and very recently to
random walks on hyperbolic groups by
\textsc{Sidoravicius -- Wang --Xiang}~\cite{Sidoravicius-Wang-Xiang20p}.

As outlined in the Introduction, our goal here is different; we are interested in
random  \emph{limit boundary measures} arising from the sequences \eqref{eq:empdis}, resp.
\eqref{eq:WW_n}. Unlike with the limit sets, the very existence of the limit measures is a non-trivial
problem. In many cases  there is a phase (regarding the branching ratio $\rho$) where the branching
Markov chain is strongly recurrent in the sense that with probability 1, each state $x \in \Xs$ is visited
by the population infinitely often, see the references of the present subsection. Nevertheless,
the empirical distributions \emph{always} move their mass to infinity, as the following lemma shows,
providing a simple motivation for our goals.

\begin{lem}\label{lem:disappear}
 Under assumptions \eqref{ass:br} and \eqref{ass:tc}, for any $y \in \Xs$
 $$
 \lim_n \dol M_n(y) = 0 \quad \text{almost surely.}
 $$
\end{lem}

\begin{proof} It suffices to prove this for the situation when the branching chain starts with
 one particle at a generic $x \in \Xs$. In view of \eqref{eq:transprob},
$$
\pmb\Ee_x\bigl(M_n(y)\bigr) = p^{(n)}(x,y) \, \rho^n\,.
$$
Therefore
$$
\pmb\Ee_x\biggl(\sum_n \frac{1}{\rho^n} M_n(y)\biggr) = G(x,y) < \infty\,.
$$
Therefore $M_n(y)/\rho^n \to 0$ and thus also $\dol M_n(y)\to 0$ almost surely under $ \pmb\Pe_{\!x}\,$.
\end{proof}

\section{Uniform integrability and positivity of the population martingale} \label{sec:pm}

\subsection{The population martingale} \label{subsec:mart}

Recall that we assume \eqref{ass:br}: the offspring averages satisfy $\ol\pi_x=\rho < \infty$ for all
$x\in\Xs$. In terms of the augmentation map \eqref{eq:aug}, the barycentre map \eqref{eq:ol}, and the transition operator $\Pc$ \eqref{eq:pe} this condition means that
$$
\bigl\| \ol{\Th\Pc} \bigr\| = \rho \bigl\| \ol\Th \bigr\| \qquad\forall\,\Th\in\Prob(\Ms) \;,
$$
or, equivalently,
$$
\ol{\aug(\Th\Pc)} = \rho \cdot \ol{\aug(\Th)} \qquad\forall\,\Th\in\Prob(\Ms) \;.
$$
In other words, after one step of the branching Markov chain the average size of populations is always
multiplied by the same constant $\rho$. This is the case, for instance, for the branching Markov chains
from examples \ref{ex:bd}, \ref{ex:bs}, and \ref{ex:gr}; in the setup of \exref{ex:gm}, condition \eqref{ass:br} was used by \textsc{Gantert -- M\"uller} \cite[Section~3.1]{Gantert-Muller06}.
Recall the Definition \ref{dfn:pop} of the population martingale.  The sequence $(W_n)$ is indeed a
martingale with respect to the increasing coordinate filtration on the path space, because by
\corref{cor:har} condition \eqref{ass:br} implies that the lift $\wt\1(m) = \|m\|$ \eqref{eq:1} of the
constant function $\1$ from $\Xs$ to $\Ms$ is $\rho$-harmonic; see \secref{subsec:hm} below for a more
general discussion.

\emph{A priori} the expectation of the limit population ratio
$$
\pmb\Ee_\Th W_\infty = \langle \, \pmb\Pe_{\!\Th}, W_\infty \rangle
$$
may be \emph{strictly smaller} than the expectations
$$
\pmb\Ee_\Th W_n = \pmb\Ee_\Th W_0 = \sum_{m \in \Ms} \Th(m)\,\|m\| 
$$
of the population martingale with respect to the measure $\pmb\Pe_{\!\Th}$ on the path space corresponding to an initial distribution $\Th\in\Prob(\Ms)$. Their equality means that the population martingale is \emph{uniformly integrable} on the path space $(\Ms^\Zp,\pmb\Pe_{\!\Th})$ (e.g., see \textsc{Meyer} \cite[Chapter V]{Meyer66} for the basics of martingale theory). When talking about uniform integrability without specifying a measure on the path space we mean that it holds for \emph{any} initial distribution $\Th\in\Prob(\Ms)$, i.e., with respect to the full initial support measure class \eqref{fs}. In order to guarantee this property it is enough to take for $\Th$ just the delta measures concentrated at singletons $\de_x,\;x\in\Xs$, i.e., to require that
$$
\pmb\Ee_x W_\infty = 1 \qquad\forall\, x\in\Xs \,.
$$

\subsection{Uniform $L\log L$ moment condition} \label{sec:ll}

For the ordinary Galton -- Watson processes (\exref{ex:GW}) the equivalence of the uniform integrability of the population martingale to the \sfemph{$L\log L$ moment condition}
\begin{equation} \label{eq:log}
\sum_k \pi(k)\cdot k\log k < \infty \;.
\end{equation}
on the offspring distribution $\pi$ is the classical theorem of
\textsc{Kesten -- Stigum}~\cite{Kesten-Stigum66} (see also
\textsc{Lyons~-- Pemantle~-- Peres}~\cite{Lyons-Pemantle-Peres95a} and the references therein). 
Although this criterion is directly applicable to the situation when the offspring distributions 
$\pi_x$ are the same for all $x\in\Xs$, in particular, to branching random walks on groups 
(see \exref{ex:gr}), this is not the case for general branching Markov chains.

In order to formulate an analogous result in the general setup we need \emph{tightness} of the
offspring distributions, as follows.

\begin{dfn}\label{dfn:tight}
Given two probability distributions $\pi$ and $\pi'$ on $\Zp$, we say that $\pi$ 
\sfemph{dominates} $\pi'$ (notation: $\pi'\preceq \pi$) if
$$
\pi'[n,\infty) \le \pi[n,\infty) \qquad\forall\,n\in\Zp \,.
$$
A family of probability measures on $\Zp$ satisfies the \sfemph{uniform first moment condition} (resp.,
the \sfemph{uniform $L\log L$ moment condition}) if it is dominated by a probability measure with a finite
first moment (resp., by a measure that satisfies the $L\log L$ moment condition).
\end{dfn}

The uniform moment condition was used, for example,
by \textsc{Kaimanovich~-- Woess} \cite[Lemma 1]{Kaimanovich-Woess92} for random walks on graphs, and
by \textsc{D'Souza -- Biggins}~\cite[p.~40]{D'Souza-Biggins92} for branching processes.

\begin{thm} \label{thm:ll}
If the offspring distributions of a branching Markov chain satisfy the uniform $L\log L$ moment condition, then the population martingale is uniformly integrable.
\end{thm}

It is known since \textsc{Levinson} \cite[Section 4]{Levinson59} that for the ordinary
Galton -- Watson processes the $L\log L$ condition implies that the limit population
ratio is almost surely strictly positive on non-extinction. A consequence of the
Kesten -- Stigum theorem is the equivalence (on non-extinction) of the following two conditions:
\begin{enumerate}[{\rm (i)}]
\item the population martingale $(W_n)$ is uniformly integrable;
\item the limit population ratio $W_\infty$ is almost surely strictly positive.
\end{enumerate}

However, for branching processes in varying environment it may well happen that the limit population ratio
vanishes with positive probability in spite of the uniform integrability of the population martingale (see
the example constructed in \textsc{MacPhee -- Schuh} \cite{MacPhee-Schuh83} and the discussion in
\textsc{D'Souza -- Biggins} \cite[p.~41]{D'Souza-Biggins92}). We do not know whether in our setup the
uniform integrability of the population martingale would always imply that the limit population ratio is
almost surely positive. 
Still, we can show that this is the case under the same uniform $L\log L$ condition as in \thmref{thm:ll}.

\begin{thm} \label{thm:pos}
If the offspring distributions of a branching Markov chain satisfy the uniform $L\log L$ moment condition, then the limit population ratio is almost surely strictly positive for any initial population.
\end{thm}

Our proofs of \thmref{thm:ll} and \thmref{thm:pos} below are self-contained and follow the approach of \textsc{D'Souza~-- Biggins} \cite{D'Souza-Biggins92} to the Galton -- Watson processes in varying environment. \thmref{thm:ll} can also be deduced from the general criterion of uniform integrability of the martingales of multi-type branching processes ($\equiv$ branching Markov chains in our terminology)
associated with ``mean-harmonic functions'' due to \textsc{Biggins~-- Kyprianou} \cite[Theorem 1.1 and the discussion on p. 547]{Biggins-Kyprianou04}, cf.\ \remref{rem:hm} below.

\subsection{Laplace transforms and their remainders}

We denote by
$$
\Gc_\th(s) = \sum_k \th(k) e^{-sk} = \Eb_{\th} e^{-sX}
$$
the \sfemph{Laplace transform} of a probability measure $\th\in\Prob(\Zp)$ (which can be thought of as
the distribution of a $\Zp$-valued random variable $X$). The linear part of the power series expansion
of $\Gc_\th(e^{-s})$ is equal to $1-\ol\th s$, where $\ol\th$ is the expectation of $\th$ (assumed to be finite), and we denote the arising \sfemph{remainder} by
\begin{equation} \label{eq:r}
\Rc_\th(s) = \Gc_\th(s) - 1 + \ol\th s
= \sum_k \th(k) \, \psi(s k)
= \Eb_{\th} \psi(sX)
\end{equation}
with
\begin{equation} \label{eq:psi}
\psi(t) = e^{-t} - 1 + t \ge 0 \;.
\end{equation}
We also use the above notation with the subscript $\Th$ in the situation when
$\th=\aug(\Th)$
is the image of a measure $\Th\in\Prob(\Ms)$ under the augmentation map $\aug$ \eqref{eq:aug}, so that
$$
\Gc_\Th(s) = \pmb\Ee_{\Th} e^{-s\|M\|} \;, \qquad \Rc_\Th(s) = \pmb\Ee_{\Th} \psi(s\|M\|) \,.
$$
\begin{lem} \label{lem:rc}
For any measure $\th\in\Prob(\Zp)$ with a finite first moment
\begin{enumerate}[{\rm (i)}]
\item
the function $\Rc_\th$ is non-decreasing on the positive ray $\RR_+$;
\item
the ratio $\Rc_\th(s)/s$ is non-decreasing on $\RR_+$, and
$$
\lim_{s\to 0} \frac{\Rc_\th(s)}{s}= 0 \;;
$$
\item
the integral
$$
\int_0^C \frac{\Rc_\th(s)}{s^2} \,ds
$$
is convergent for any 
$C>0$ if and only if the measure $\th$ satisfies the $L\log L$ moment condition \eqref{eq:log}.
\end{enumerate}
Further, if a measure $\th\in\Prob(\Zp)$ with a finite first moment dominates another measure
$\th'\in\Prob(\Zp)$, then
\begin{enumerate}[{\rm (iv)}]
\item
\begin{equation} \label{eq:rc}
\Rc_{\th'}(s) \le \Rc_\th (s) \qquad\forall\, s\ge 0 \;.
\end{equation}
\end{enumerate}
\end{lem}

\begin{proof}
(i) and (ii) immediately follow from the same properties of the functions $\psi$ \eqref{eq:psi} and $s\mapsto\psi(s)/s$, respectively, whereas (iv) is a consequence of (i). Property (iii) is well-known, e.g., see \textsc{Athreya -- Ney} \cite[Lemma~I.10.1]{Athreya-Ney72}. Since our setup is somewhat different, for the sake of completeness we include its elementary proof.

The function $\Rc_\th$ being non-negative, by exchanging the order of summation and integration one arrives at
$$
\int_0^C \frac{\Rc_\th(s)}{s^2} \,ds
= \int_0^C \sum_{k=0}^\infty \frac{\psi(sk)}{s^2} \th(k) \,ds
= \sum_{k=0}^\infty \th(k) \int_0^C \frac{\psi(sk)}{s^2} \,ds \;,
$$
where
$$
\int_0^C \frac{\psi(sk)}{s^2} \,ds
= k \int_0^{kC} \frac{\psi(s)}{s^2}\,ds \;.
$$
Since $\psi(s)/s\to 1$ as $s\to\infty$, the latter integral asymptotically behaves as $k\log (kC)$, whence the claim.
\end{proof}

\begin{lem} \label{lem:pipe}
If the offspring distributions $\pi_x$ of a branching Markov chain satisfy the uniform first moment condition, then there exists $s_0>0$ such that for any measure $\Th\in\Prob(\Ms)$
\begin{equation} \label{eq:pipe}
\Gc_{\Th\Pc}(s) \le \Gc_\Th(\rho s) + \ol\th\, \Rc(s) \qquad\text{for all }\;s\in [0,s_0]\;,
\end{equation}
where $\rho$ is the common branching ratio from condition \eqref{ass:br}, $\ol\th$ is the expectation of the measure $\th=\aug(\Th)$, and $\Rc=\Rc_\pi$ is the remainder function \eqref{eq:r} associated with the measure $\pi\in\Prob(\Zp)$ that dominates the distributions $\pi_x$.
\end{lem}

\begin{proof}
To begin with, let $\Th$ be the delta measure at the singleton $\de_x \in \Ms,\; x\in\Xs$. Then
$\Th\Pc=\Pi_x$, see \dfnref{dfn:bmc}, whence $\aug(\Th\Pc)=\pi_x$. 
We recall that $\ol\pi_x=\rho$ for all $x\in\Xs$ by our standing assumption \eqref{ass:br}. Therefore, by \lemref{lem:rc}(iv) for any $s\ge 0$
\begin{equation} \label{eq:es}
\begin{aligned}
\Gc_{\Th\Pc}(s) = \Gc_{\pi_x}(s)
&= 1 - \ol\pi_x s + \Rc_{\pi_x}(s) \\
&= 1 - \rho s + \Rc_{\pi_x}(s)
\le 1 - \rho s + \Rc(s)\;.
\end{aligned}
\end{equation}

Now, let $\Th=\de_m$ for $m\in\Ms$, so that $\Th\Pc=\Pi_m$. Then by \eqref{eq:PiM} and \eqref{eq:es},
$$
\Gc_{\Th\Pc}(s)
= \Gc_{\Pi_m}(s)
= \prod_{x\in m} \Gc_{\pi_x}(s)
\le (1 - \rho s + \Rc(s))^{\|m\|}
$$
(counting as always multiplicities in the product).
By \lemref{lem:rc}(ii) we can choose $s_0>0$ in such a way that
$$
\Rc(s) \le \rho s \le 1 \qquad \forall\,s\in [0,s_0] \;.
$$
Since the derivative of the function $t\mapsto t^{\|m\|}$ on the interval $[0,1]$ does not exceed $\|m\|$, we then have
$$
\begin{aligned}
(1 - \rho s + \Rc(s))^{\|m\|}
&\le (1 - \rho s)^{\|m\|} + \|m\|\Rc(s) \\
&\le e^{-\rho s \|m\|} + \|m\|\Rc(s) \qquad\qquad\forall\, s\in [0,s_0] \;,
\end{aligned}
$$
and therefore \eqref{eq:pipe} is satisfied, because $\th=\aug(\Th)=\de_{\|m\|}$, so that $\Gc_\Th(z)=z^{\|m\|}$ and $\ol\th=\|m\|$.

Finally, the general case follows from the linearity of the both sides of \eqref{eq:pipe} with respect to $\Th$.
\end{proof}

\subsection{Proof of \thmref{thm:ll}} \label{subsec:proof-ll}

We denote by
$$
\Th_n = \de_{\de_x} \Pc^n \,, \qquad t\in\Zp \,,
$$
the one-dimensional distributions of the associated measure $\pmb\Pe_{\!x}$ on the space of sample paths of the branching Markov chain $\Mb=(M_0,M_1,\dots)$ with $M_0=\de_x \in \Ms$. Then
$$
\pmb\Ee_x\, e^{-s W_\infty}
= \lim_n \pmb\Ee_x\, e^{-s W_n}
= \lim_n \pmb\Ee_x\, e^{-s \,\|M_n\|/\rho^n}
= \lim_n \Gc_{\Th_n} \Bigl(s/\rho^n) \;.
$$
Condition \eqref{ass:br} implies that
$$
\ol{\aug(\Th_n)} = \rho^n  \;,
$$
whence by \lemref{lem:pipe} for $s\le s_0$
$$
\Gc_{\Th_n} (s/\rho^n) \le \Gc_{\Th_{n-1}} (s/\rho^{n-1})
+ \rho^{n-1} \Rc (s/\rho^n) \,,
$$
and by telescoping
\begin{equation} \label{eq:ineq}
\begin{aligned}
\pmb\Ee_x\, e^{-s W_\infty}
&\le \Gc_{\Th_0} (s)
+ \sum_{n=1}^\infty \rho^{n-1} \Rc(s/\rho^n) \\
&\le e^{-s}
+ \frac{1}{\rho} \int_0^\infty \rho^\tau \Rc(s/\rho^\tau)\,d\tau
= e^{-s} + \frac{s}{\rho\log\rho} \int_0^s \frac{\Rc(\si)}{\si^2}\,d\si \,.
\end{aligned}
\end{equation}
Thus,
$$
\begin{aligned}
\pmb\Ee_x\, W_\infty
&= \lim_{s\to 0} \frac{1 - \pmb\Ee_x\, e^{-s W_\infty}}{s} \\
&\ge \lim_{s\to 0} \left\{
\frac{1-e^{-s}}{s} - \frac{1}{\rho\log\rho} \int_0^s \frac{\Rc(\si)}{\si^2}\,d\si \right\}
= 1 \,.
\end{aligned}
$$

\subsection{Proof of \thmref{thm:pos}} \label{subsec:proof-pos}

Let
\begin{equation} \label{eq:om}
\om(m) = \pmb\Pe_m [W_\infty=0]
\end{equation}
denote the probability that the limit population ratio of the branching random walk issued from an initial population $m\in\Ms$ vanishes. Somewhat abusing notation, we also put $\om(x) = \om(\de_x)$ if the initial population is the singleton $\de_x$ at a point $x\in\Xs$.

The function $\om$ on~$\Ms$ is $\Pc$-harmonic, its values are sandwiched between $0$ and $1$, and
$$
\om(m_1+m_2) = \om(m_1) \om(m_2) \qquad\forall\,m_1,m_2\in\Ms \;.
$$
This implies that the function $\om$ is determined by its values on singletons as
$$
\om(m) = \prod_{x\in m} \om(x) \;,
$$
where, as always, 
each point from the support of~$m$ is taken with its multiplicity.

By inequality \eqref{eq:ineq} from the proof of \thmref{thm:ll}
$$
\om(x) = \pmb\Pe_{\!x} \{W_\infty=0\}
\le \pmb\Ee_x\, e^{-sW_\infty}
\le e^{-s} + \frac{s}{\rho\log\rho} \int_0^s \frac{\Rc(\si)}{\si^2}\,d\si \,,
$$
with the right-hand side of this inequality being strictly less than 1 for all sufficiently small~$s$ in view of \lemref{lem:rc}(iii). Therefore, the function $\om$ \eqref{eq:om} is bounded away from~$1$, i.e., there exists $c<1$ such that
$$
\om(x) \le c \qquad\forall\,x\in \Xs\;.
$$
The fact that the offspring distributions $\pi_x$ satisfy the uniform first moment condition,
whereas their expectations are equal to $\rho>1$, implies that the probabilities
$\pi_x[2,\infty)$ are bounded away from 0. Therefore, at each step of the branching
Markov chain the size of the population increases with a probability bounded away
from $0$, so that $\|M_n\|\to\infty$ almost surely.
Thus,
$$
\om(M_n) \le \eta^{\|M_n\|} \to 0 \;.
$$
We have already mentioned that the function $\om$ is $\Pc$-harmonic, whence $\om\equiv 0$.

\section{Topological convergence of populations} \uplabel{sec:tc}

\subsection{Harmonic systems of measures and stationary spaces} \label{subsec:hars}
In this and the next subsection, we set up the needed background on boundary behaviour
for transient Markov chains, to be applied to the base Markov chain of our branching
chain and subsequently the branching Markov chain itself.

The action of the transition operator of a countable state space Markov chain (see \secref{subsec:trans})
naturally extends from the ``ordinary'' real valued functions to the ones taking values in an arbitrary
affine space (provided infinite convex combinations are well-defined\,---\,this is needed if not all
transition probabilities are finitely supported), in particular, to measure valued functions. By
$\Prob(\Ks)$ we denote the space of probability measures on a measurable space $\Ks$, and in the same way as for real functions we can formulate

\begin{dfn} \label{dfn:stat}
A map
\begin{equation} \label{eq:kadef}
\ka:\Xs\to\Prob(\Ks) \;,\qquad  x\mapsto\ka_x \;,
\end{equation}
---\,in other words, a system $(\ka_x)$ of probability measures on $\Ks$ indexed by a countable space~$\Xs$\,---,is \sfemph{harmonic} 
with respect to a Markov operator $P$ on $\Xs$, if 
$
P\ka=\ka,
$
i.e., if $\ka$ satisfies the mean value property
\begin{equation} \label{eq:kax}
\ka_x = \langle \mu_x,\ka \rangle = \sum_y p_x(y)\, \ka_y \qquad\forall\,x\in \Xs \,,
\end{equation}
where $p_x$ are the transition probabilities of the operator $P$. One also uses the term
\sfemph{stationary} (or, \sfemph{$P$-stationary}), cf.\ \remref{rem:stat} and \exref{ex:stat}.

We shall refer to the couple $(\Ks,\ka)$ as a \sfemph{measurable $P$-stationary space}. In the situation
when $\Ks$ is a topological space endowed with the Borel sigma-algebra, we call it a \sfemph{topological
$P$-stationary space}.
\end{dfn}

As follows from our irreducibility assumption \eqref{ass:tc}, all measures $\ka_x$ in a harmonic system
$\ka=(\ka_x)$ are pairwise equivalent. Therefore one can talk about their common measure class
and we use the notation $L^\infty(\Ks,\ka)=L^\infty(\Ks)$ for the corresponding Banach space of essentially bounded measurable functions.


\begin{rem} \label{rem:stat}
Given a map $\ka$ \eqref{eq:kadef}, we use the same notation for its extension
\begin{equation} \label{eq:kamu}
\ka_\th = \sum_x \th(x)\,\ka_x \;, \qquad \th\in\Meas(\Xs) \;.
\end{equation}
Then the $P$-harmonicity of a map $\ka$ is equivalent to its invariance with respect to the 
action of the operator $P$ on $\Meas(\Xs)$, that is,
$\ka_{\th} = \ka_{\th P}$ for all $\th\in\Meas(\Xs) \,$.

The dual statement is that \eqref{eq:kax} holds if and only if
for any test function $\ph\in L^\infty(\Ks)$, or from
the Banach space $C(\Ks)$ of real valued continuous functions when $\Ks$ is a topological space,
the function
n the measurable case
\begin{equation} \label{eq:test}
f^\ph(x) = \langle \ka_x, \ph \rangle \;, \qquad x\in\Xs \;,
\end{equation}
is $P$-harmonic in the usual sense. In particular, a \emph{non-constant} harmonic system 
exists only if there are non-constant bounded $P$-harmonic functions on $\Xs$.
\end{rem}

\begin{prp} \label{prp:c}
If $(\Ks,\ka)$ is a compact separable $P$-stationary space, then with probability $1$, the Markov chain
$\Xb=(X_n)$ has a random weak* limit
$$
\kb_\Xb = \wlim_{n\to\infty} \ka_{X_n} \in \Prob(\Ks) \;,
$$
and the barycentre of the arising family of measures $\{\kb_\Xb\}$ on $\Ks$ with respect to any distribution $\Pb_{\!x},\;x\in\Xs$, on the path space is the measure
\begin{equation} \label{eq:ikbxb}
\Eb_x(\kb_\Xb) = \ka_x\,\quad \text{i.e.,} \quad \Eb_x(\langle \kb_{\Xb}\,, \ph \rangle )
= \langle \ka_x\,, \ph \rangle \quad  \forall \ph \in C(\Ks)\,, x \in \Xs.
\end{equation}
\end{prp}

\begin{proof}
As we have already mentioned, for any $\ph\in C(\Ks)$ the function $f^\ph$ \eqref{eq:test} 
on $\Xs$ is $P$-harmonic and obviously bounded. Therefore, the sequence of its values 
$f^\ph(X_n)$ along the sample paths of the chain is a bounded martingale with respect 
to the coordinate filtration of the path space (see \secref{sec:tp} below for more details), 
whence the limit
\begin{equation} \label{eq:lim}
\Kb (\ph) = \lim_n f^\ph(X_n) = \lim_n \langle \ka_{X_n}, \ph \rangle
\end{equation}
exists for almost every sample path, and
\begin{equation} \label{eq:lim2}
\Eb_x\bigl(\Kb(\ph)\bigr) = \Eb_x\bigl( f^\ph(X_0)\bigr) = f^\ph(x)
= \langle \ka_x, \ph \rangle
\quad\forall\,x\in\Xs \;.
\end{equation}
If $\Phi\subset C(\Ks)$ is a countable dense subset, then by discarding the exceptional 
sets for each function $\ph\in\Phi$ one obtains a co-negligible subset $\Om$ of the path 
space such that the limit \eqref{eq:lim} exists 
on $\Om$ for \emph{all} functions $\ph\in\Phi$, hence, by the density assumption, for
\emph{all} $\ph\in C(\Ks)$. Hence, 
this limit determines a non-negative normalised linear functional on $C(\Ks)$,
i.e., a Borel probability measure $\kb_\Xb$ on $\Ks$ such that
\begin{equation} \label{eq:kaxb}
\int \ph\,d\kb_\Xb = \Kb(\ph) \quad \text{on }\; \Om\,,
\end{equation}
and convergence in \eqref{eq:lim}
$$
\langle \ka_{X_n},\ph \rangle \to \Kb (\ph)\qquad\forall\,\ph\in C(\Ks)
$$
is precisely the weak$^*$ convergence of the measures $\ka_{X_n}$ to $\kb_\Xb$ on 
$\Om$.
Now, in terms of the limit measures $\ka_\Xb$ \eqref{eq:kaxb} formula \eqref{eq:lim2} takes the form
$$
\Eb_x\bigl( \langle \kb_\Xb,\ph \rangle\bigr)= \langle \ka_x, \ph \rangle \qquad\forall\,\ph\in C(\Ks) \;,
$$
which proves the statement on the barycentre.
\end{proof}

\begin{ex} \label{ex:stat}
Let $\Xs$ be a countable group continuously acting on a compact space~$\Ks$. Given a probability measure
$\mu$ on $\Xs$, a measure $\ka$ on $\Ks$ is called \sfemph{$\mu$-stationary} if it is preserved by the
\sfemph{convolution} with $\mu$:
\begin{equation} \label{eq:ka}
\ka = \mu * \ka = \sum_x \mu(x)\, \ka_x \;,
\end{equation}
where
$$
\ka_x(A)=\ka(x^{-1}A) \;, \qquad A\subset\Ks \;,
$$
is the \sfemph{$x$-translate} of the measure $\ka$. Since $\Ks$ is compact, the existence of $\mu$-stationary probability measures is guaranteed by the Krylov~-- Bogolyubov theorem
$$
\ka\mapsto\mu*\ka \;,
$$
see \textsc{Furstenberg} \cite[Definition 1.2 and Lemma 1.2]{Furstenberg63a}. In terms 
of the transition operator $P=P_\mu$ of the random walk on $\Xs$ determined by $\mu$ 
(cf.\ \exref{ex:gr}), the $\mu$-stationarity of a measure~$\ka$ is equivalent to the 
$P$-harmonicity of the family of translates $\ka_x\,$.
The proof of \prpref{prp:c} above follows the group case argument in \textsc{Furstenberg} 
\cite[Lemma 3.1 and the ensuing Corollary]{Furstenberg71} which essentially goes back to
\textsc{Furstenberg} \cite[Lemma 1.3]{Furstenberg63a}; also see 
\textsc{Woess} \cite[Theorem~2.2]{Woess96}, \cite[Theorem~20.3]{Woess00} 
(cf.\ \prpref{prp:dir} below).
\end{ex}

Extending the notion of a \sfemph{$\mu$-boundary} for random walks on groups introduced by \textsc{Furstenberg} \cite[Section 8]{Furstenberg73}, at this point we formulate the following.

\begin{dfn} \label{dfn:pb}
A compact separable $P$-stationary space is a \sfemph{topological $P$-boundary} if
with probability 1, the random limit measure $\kb_\Xb$ is a delta measure at a random point.
\end{dfn}

As we shall see in \prpref{prp:bdry} below, topological $P$-boundaries considered as 
measure spaces can be characterised as quotients of the \emph{Poisson boundary} of the chain $(\Xs,P)$.

\subsection{Compactifications and the Dirichlet problem} \label{subsec:cdp}

\emph{A priori} the stationary space~$\Ks$ in \dfnref{dfn:stat} and \prpref{prp:c} does not have to be ``attached'' to the state space~$\Xs$ in any way. Let us now look at the situation when $\Ks$ is the \emph{boundary} $\p\Xs$ of a \emph{compactification} $\ol\Xs= \Xs \cup \p\Xs$ of the state space $\Xs$.
\begin{equation} \label{ass:sc} \tag{\sfemph{SC}}
\cond{We only consider 
compactifications for which 
$\p\Xs$ is separable.}
\end{equation}
\note{W: I do not like ``separable compactification'' when it is the boundary which is separable.}
(Since $\Xs$ is countable, the compactification space $\ol\Xs$ is always separable. Still, the boundary
$\p\Xs$ need not be separable in general,---\,such as, e.g., the \emph{Stone -- {\v C}ech
compactification}).

\begin{dfn} \label{dfn:sr}
A compactification of the state space $\Xs$ of a Markov chain is \sfemph{stochastically resolutive} 
with respect to this chain if $X_n$ converges 
almost surely to the compactification boundary, i.e., with probability $1$
there exists the limit
$$
X_\infty = \lim_n X_n \in \p\Xs \;.
$$
The resulting images $\ka_x$ of the measures $\Pb_{\!x}$ under the limit map are called the \sfemph{hitting distributions} of the Markov chain.
\end{dfn}

This definition alludes to the notion of \emph{resolutivity} from classical potential theory (e.g., see \textsc{Luke\v{s}~-- Netuka - Vesel\'{y}} \cite[Section~4]{Lukes-Netuka-Vesely02}), cf.\ the remark at the beginning of Section 8 in \textsc{Woess} \cite{Woess96}.
By the Markov property the system of hitting measures $\ka_x$ of a stochastically resolutive compactification is $P$-harmonic in the sense of \dfnref{dfn:stat}.

\begin{prp}
The boundary $\p\Xs$ of a stochastically resolutive compactification endowed with the family of the
hitting measures $(\ka_x)$ is a $P$-boundary, and
$$
\wlim_{n\to\infty} \ka_{X_n} = \de_{X_\infty} \qquad \text{almost surely.}
$$
\end{prp}

This is a consequence of a general result on the identification of $P$-boundaries with the quotients of the Poisson boundary (\prpref{prp:bdry}) which we relegate to \secref{sec:tp}.

\begin{dfn} \label{dfn:dir}
A compactification of the state space $\Xs$ of a Markov chain with the transition operator $P$ is
\sfemph{Dirichlet regular} with respect to this chain if for any function $\ph\in C(\p\Xs)$ there
is a unique $P$-harmonic function $f^\ph$ on $\Xs$ (the \sfemph{solution of the Dirichlet problem}
with the boundary data $\ph$) that provides a continuous extension of $\ph$ to all of $\ol\Xs$. In
this situation for any $x\in\Xs$
$$
\ph \mapsto f^\ph(x)
$$
is a norm 1 positive linear functional on $C(\p\Xs)$ represented by a Borel probability
measure~$\ka_x$ on $\p\Xs$ ($\equiv$ the \sfemph{harmonic measure} with pole at $x$) as
\begin{equation} \label{eq:kah}
f^\ph(x) = \int \ph(\xi)\,d\ka_x(\xi) = \langle \ka_x, \ph \rangle \;.
\end{equation}
\end{dfn}

The system of harmonic measures from \dfnref{dfn:dir} is $P$-harmonic in the sense of 
\dfnref{dfn:stat} (cf.\ \remref{rem:stat}).

\begin{prp}[\textsc{Woess} {\cite[Theorem 2.2]{Woess96}, \cite[Theorem 20.3]{Woess00}}] \label{prp:dir}
A compactification satisfying \eqref{ass:sc} $\ol\Xs=\Xs\cup\p\Xs$  of the state space 
$\Xs$ of a transient Markov chain is Dirichlet regular if and only if the following two conditions hold:
\begin{enumerate}[{\rm (i)}]
\item
the compactification is stochastically resolutive;
\item
the system of the hitting measures $\ka_x$ has the property that
$$
\wlim_{x \to \xi} \ka_x = \de_\xi \qquad\forall\,\xi\in\p\Xs \;.
$$
\end{enumerate}
In this situation the measures arising from the solvability of the Dirichlet problem coincide
with the hitting measures $\ka_x$.
\end{prp}

\begin{rem}
In terms of the boundary convergence the difference between stochastic resolutivity and Dirichlet
regularity is that in the latter case the harmonic measures $\ka_{x_n}$ converge to the delta measure
$\de_{x_\infty}$ at the limit point $x_\infty=\lim x_n \in \p\Xs$ for \emph{any} boundary convergent
sequence $(x_n)$, whereas in the latter case $\ka_{X_n}\to \de_{X_\infty}$ just almost surely, i.e., along \emph{almost all} sample paths 
of the chain. Stochastic resolutivity on its own does by no means imply Dirichlet regularity. For instance, see \textsc{Benjamini -- Peres} \cite[Example~3]{Benjamini-Peres92} and \textsc{Kaimanovich~-- Woess} \cite[pp.\ 461-462]{Kaimanovich-Woess92} for examples of this kind with random walks on trees.
\end{rem}

\begin{rem}
If a $P$-stationary space $(\Ks,\ka)$ is compact, then the map $\ka:x\mapsto\ka_x$ 
\eqref{eq:kadef} provides an \emph{embedding} of the discrete space~$\Xs$ into the 
compact space $\Prob(\Ks)$ of Borel probability measures on $\Ks$ endowed with the 
weak* topology, and therefore it gives rise to a \emph{compactification} of $\Xs$ 
whose boundary is the collection of all weak* limit points of the system $(\ka_x)$. 
This idea goes back to \textsc{Furstenberg} \cite[Chapter II]{Furstenberg63} who 
used it to define a compactification of Riemannian symmetric spaces. \prpref{prp:c} 
then implies that this compactification is stochastically resolutive.
\end{rem}

Our various preliminary considerations lead to the following, which is going to be a basic 
tool for proving a.s.\ convergence of the empirical distributions
to a random distribution on the boundary.

\begin{prp} \label{prp:cc}
Let $\ol\Xs=\Xs\cup\Xs$ be a Dirichlet regular compactification of the state space $\Xs$ of a transient Markov chain. If
$$
\wlim_{n\to\infty} \th_n = \th_\infty \in \Prob(\p\Xs) \;,
$$
for a sequence of measures $\th_n\in\Prob(\Xs)$, then also
$$
\wlim_{n\to\infty} \ka_{\th_n} = \th_\infty
$$
for the sequence of the associated harmonic measures $\ka_{\th_n}$.
\end{prp}

\begin{proof}
Let $\ph\in C(\p\Xs)$ be a continuous test function on $\p\Xs$, and let $f^\ph\in C\left(\ol\Xs\right)$ be its harmonic extension to the whole of 
$\ol\Xs$. Then by the definitions of the harmonic measures and of the weak* convergence
$$
\langle \ka_{\th_n}, \ph \rangle
= \langle \th_n , f^\ph \rangle
\to \langle \th_\infty, f^\ph \rangle
= \langle \th_\infty, \ph \rangle \;,
$$
whence the claim.
\end{proof}

\begin{cor} \label{cor:cc}
Under the conditions of \prpref{prp:cc}, let $\th_n\in\Prob(\Xs)$ be a sequence of measures 
escaping to infinity on $\Xs$ (i.e., such that $\th_n(x)\to 0$ for any $x\in\Xs$). 
Then the sequence $\th_n$ converges if and only if the sequence of the harmonic 
measures $\ka_{\th_n}$ converges, and the limits of these two sequences coincide.
\end{cor}

\begin{proof}
The claim follows from the compactness of the space $\Prob\left(\ol\Xs\right)$ in the 
weak* topology. Indeed, if $\ka_{\th_n}$ is convergent, whereas $\th_n$ is not, then 
by the compactness the sequence~$\th_n$ has at least two different limit measures 
$\th_\infty^1,\th_\infty^2$ which by the escape assumption are supported by $\p\Xs$. 
By taking sub-sequences of $\th_n$ converging to $\th_\infty^1$ and to $\th_\infty^2$, 
respectively, one then arrives at a contradiction with \prpref{prp:cc}.
\end{proof}

\subsection{Population convergence}

We finally come to the application, resp. extension of the results from subsection \ref{subsec:hars}
and subsection \ref{subsec:cdp} to the setup of branching Markov chains (see subsection \ref{subsec:bd}).
We recall that, given a map $\ka:\Xs\to\Prob(\Ks)$, we denote by
$$
\ka_m = \sum_{x\in m} \ka_x \;, \qquad m\in\Ms \;,
$$
its extension \eqref{eq:kamu} to the population space $\Ms=\Zp(\Xs)$ over $\Xs$, so that, in particular,
\begin{equation} \label{eq:kam}
\| \ka_m \| = \| m \| \qquad\forall\, m\in \Ms \;.
\end{equation}
The \sfemph{normalisation} 
\begin{equation}\label{eq:nor}
\dol{\ka}_m = 
\frac{1}{\|m\|} \, \ka_m\in \Prob(\Ks)
\end{equation}
is then the \emph{average} of the measures $\ka_x$ over a population $m$ (where $m$, as always, is treated as a multiset).

\begin{thm} \label{thm:kconv}
If
\begin{enumerate}
\item[{\rm (1)}]
a branching Markov chain on the state space $\Xs$ has constant branching ratio $\rho>1$, and
its offspring distributions satisfy the uniform $L\log L$ moment condition,
\item[{\rm (2)}]
$(\Ks,\ka)$ is a separable compact stationary space for the underlying Markov chain on~$\Xs$,
\end{enumerate}
then
\begin{enumerate}
\item[{\rm (I)}]
for almost every sample path $\Mb=(M_n)$ of the branching Markov chain there exists the limit
$$
\kb_\Mb = \wlim_{n\to\infty} \frac{1}{\rho^n}\, \ka_{M_n}\,,
$$
which is a positive finite Borel measure on $\Ks$;
\item[{\rm (II)}]
the barycentre of the measures $\{\kb_\Mb\}$ with respect to any distribution $\pmb\Pe_{\!x},\;x\in\Xs$, on the path space of the branching Markov chain is $\ka_x\,$:
$$
\pmb\Ee_x(\kb_\Mb) = \ka_x\,.
$$
\end{enumerate}
In particular, if
\begin{enumerate}
\item[{\rm (3)}]
$\ol\Xs=\Xs\cup\p\Xs$ is a compactification of the state space $\Xs$ which satisfies \eqref{ass:sc} and
is stochastically resolutive with respect to the underlying Markov chain,
\end{enumerate}
then {\rm (I)} and {\rm (II)} hold for the associated family of hitting distributions on
the boundary~$\p\Xs$.
Furthermore, if in addition
\begin{enumerate}
\item[{\rm (4)}]
the compactification is Dirichlet regular for the underlying Markov chain,
\end{enumerate}
then
\begin{enumerate}
\item[{\rm (III)}] For almost every sample path of the branching Markov chain,
$$
\kb_\Mb = \wlim_{n\to\infty} \frac{1}{\rho^n}\,M_n\,.
$$
\end{enumerate}
\end{thm}

\begin{proof}
The argument for the proof of (I) and (II) is essentially the same as in the proof of
\prpref{prp:c} (which could potentially be generalised to allow the measures from a harmonic
family to be not necessarily normalised and to depend on time, cf.\ \secref{subsec:hm}). The
only difference is that the arising martingales of the branching Markov chain are not bounded.
Still, they are dominated by the uniformly integrable \emph{population martingale}.

Let us first take a test function $\ph\in C(\Ks)$, let
\begin{equation} \label{eq:hh}
f(x) = f^\ph(x) = \langle \ka_x, \ph \rangle
\end{equation}
be the corresponding harmonic function of the underlying Markov chain on $\Xs$, and let
$$
\wt f(m) = \langle m, f \rangle = \langle \ka_m, \ph \rangle
$$
be its lift to $\Ms$. Then by \corref{cor:har} and \remref{rem:stat} the function $\wt f$ is $\rho$-harmonic for the branching Markov chain, whence the sequence of random variables 
$$
W_n^f
= \frac{1}{\rho^n}\,\wt f(M_n) = \frac{1}{\rho^n}\,\langle \ka_{M_n}\,, \ph \rangle
$$
on the path space of the branching Markov chain is a martingale with
\begin{equation} \label{eq:wph}
\bigl| W_n^f
\bigr| \le  \frac{1}{\rho^n}\,\| \ka_{M_n} \| \cdot \|\ph\|
=  \frac{1}{\rho^n}\,\|M_n\| \cdot \|\ph\|
= W_n
\cdot \|\ph\| \,,
\end{equation}
where $W_n = W_n(\Mb)$ is the population martingale of Definition \ref{dfn:pop}.
Then \thmref{thm:ll} on the uniform integrability of 
$(W_n)$ 
implies the uniform integrability of the martingale $(W_n^f)$ as well, so that the limit
\begin{equation} \label{eq:lw}
W_\infty^f(\Mb) = W_\infty^f
=  \lim_n W_n^f
\end{equation}
exists for almost all sample paths and has the property that
\begin{equation} \label{eq:li}
\pmb\Ee_x W_\infty^f = \pmb\Ee_x W_0^f = f(x) = \langle \ka_x, \ph \rangle
\qquad\forall\,x\in\Xs \;.
\end{equation}

If $\Phi\subset C(\Ks)$ is a countable dense subset, then there is a common co-negligible subset~$\Om$ of the path space such that the limit \eqref{eq:lw} exists and satisfies \eqref{eq:li} for all $\Mb\in\Om$ and any $\ph\in\Phi$. By \eqref{eq:wph}
$$
\bigl| W_\infty^{f_1}
- W_\infty^{f_2}
\bigr| \le W_\infty(\Mb) \cdot \| \ph_1 - \ph_2 \|
\qquad\forall\,\ph_1,\ph_2\in\Phi\,,
$$
where $f_i=f^{\ph_i},\;i=1,2,$ are the harmonic functions \eqref{eq:hh} associated with the
functions~$\ph_i$. Therefore, the limit \eqref{eq:lw} exists on $\Om$ for all $\ph\in C(\Ks)$ and
satisfies condition~\eqref{eq:li}. For any fixed realisation of $\Mb$ on $\Om$, it defines a
positive linear functional on $C(\Ks)$ whose norm is $W_\infty^\1(\Mb)=W_\infty(\Mb)$, i.e.,
a non-negative Borel measure $\kb_\Mb$ with total mass
\begin{equation} \label{eq:norm}
\|\kb_\Mb\|=W_\infty(\Mb)
\end{equation}
which is strictly positive by \thmref{thm:pos}. The identity (II) is then precisely the fact that \eqref{eq:li} is satisfied for all $\ph\in C(\Ks)$.
Finally, the existence of a stochastically resolutive compactification obviously implies the transience of the underlying chain, and therefore (I) implies (III) in view of \corref{cor:cc}.
\end{proof}

In the course of the proof of \thmref{thm:kconv} we have seen, in formula \eqref{eq:norm}, that the norm of the limit measure $\ka_\Mb$ is the limit $W_\infty(\Mb)$ of the population martingale, whence we get the following.

\begin{thm} \label{thm:emp}
Under conditions {\rm (1)} and {\rm (2)} of \thmref{thm:kconv} for almost all sample paths $\Mb=(M_n)$ of the branching Markov chain the averages
$$
\dol\ka_{\! M_n} = \frac{1}{\|M_n\|}\, \ka_{M_n} = \frac1{\|M_n\|} \sum_{x\in M_n} \ka_x
$$
converge in the weak* topology of $\Prob(\Ks)$ to the probability measure $\dol\kb_{\!\Mb}$ (the normalisation of the measure $\kb_\Mb$ from \thmref{thm:kconv}), and
$$
\ka_x = \pmb\Ee_x\bigl(W_\infty(\Mb)\cdot\dol\kb_{\!\Mb}\bigr)\,.
$$
In particular, this is the case for the boundary $\p\Xs$ of any stochastically resolutive compactification
satisfying \eqref{ass:sc} $\ol\Xs=\Xs\cup\p\Xs$ of the state space of the underlying Markov chain endowed with the family of the arising hitting measures. Moreover, if the compactification is Dirichlet regular, then the empirical distributions 
$$
\dol M_n = \frac{1}{\|M_n\|}\,M_n
$$
converge almost surely to the limit measure $\dol\kb_{\!\Mb}$ in the weak* topology of
$\Prob\left(\ol\Xs\right)$.
\end{thm}


\begin{rem}\label{rem:trivial}
 Under conditions {\rm (1)} and {\rm (2)} of \thmref{thm:kconv}, the random limit probability measure
$\dol\kb_{\!\Mb}$ is a \sfemph{random} point mass if and only if there is a \sfemph{deterministic}
element $ \kt \in \Ks$ such that $\ka_x = \delta_z$ for all $x \in \Xs$. In this case, also $\dol\kb_{\!\Mb} = \delta_z$ is deterministic.
\end{rem}

\begin{proof}
The ``if'' as well as the last statement are obvious.
For the interesting part, we need some refined notation.
We write $\Mb^x = (M_n^x)_{n \ge 0}$ for the branching Markov chain
starting at time $0$ with one particle at position $x \in \Xs$, and the
other related objects will also be equipped with the superscript $x$. In particular,
we denote by $\dol{\ka}_n^x$ the normalised measure associated with the population
at time $n$ according to \eqref{eq:nor}, that is, $\dol{\ka}_n^x = \dol{\ka}_{M_n^x}\,$.

Now let $t, n \in \Zp\,$. Then
$$
\dol{\ka}_{t+n}^x = \sum_{y \in M_t^x} \frac{\|M_n^y\|}{\|M_{t+n}^x\|}\dol{\ka}_n^y\,.
$$
Here and below, one must observe (without adding further involved notation) that the
elements $y \in M_t^x$ appear according to their multiplicity, and the respective
norms $\|M_n^y\|$ and measures $\dol{\ka}_n^y$ are independent (in particular, not identical).
If we let $n \to \infty$ and apply Theorem \ref{thm:emp} then we get
\begin{equation}\label{eq:convcomb}
\dol{\ka}_{\Mb}^x = \sum_{y \in M_t^x} \frac{W_\infty(\Mb^y)}{\rho^t\, W_\infty(\Mb^x)}
\dol{\ka}_{\Mb}^y\,
\end{equation}
and the respective limits $W_\infty(\Mb^y)$ and  limit measures $\dol{\ka}_{\Mb}^y$
are independent among themselves (including multiple appearances), but not independent
of $\Mb^x$. The sum in \eqref{eq:convcomb} is a convex combination with a.s. strictly
positive coefficients by Theorem \ref{thm:pos}.

We now take $t$ to be the first moment when $\|\Mb_t^x\| \ge 2$. By our assumptions,
this is an a.s. finite stopping time. If $\dol{\ka}_{\Mb}^x = \delta_{\zeta}$ for a
\emph{random} $\zeta \in \Ks$ then also $\dol{\ka}_{\Mb}^y = \delta_{\zeta}$ for all
$y \in M_n^x$. But the latter measures -- at least 2 -- are independent, and it
is a straightforward exercise
that $\zeta$ must be deterministic.
\end{proof}

We note that the last proposition is related to the issue of triviality of the
Poisson boundary. The latter will be considered further below.

\subsection{Adaptedness conditions} \label{subsec:adapt}

Having in mind the above boundary convergence results, we are now going to list several compactifications of the state space $\Xs$ of a discrete Markov chain (to be thought of as the underlying chain of a branching Markov chain) and comment upon the key properties of these compactifications required in \thmref{thm:kconv}: stochastic resolutivity and Dirichlet regularity.
%
Suppose that $\Xs$ carries a certain geometric, algebraic or combinatorial structure, and that the
transition operator $P$ is adapted in some way (to be specified in more detail) to that
structure. In this situation the Markov chain is usually called \emph{random walk} (so
that the corresponding branching Markov chain becomes a \emph{branching random walk},
cf.\ \exref{ex:gr}.)
How does its adaptedness affect the behaviour of the chain?
For the next considerations, we assume that $\Xs$ carries the structure of an unoriented
infinite graph which is
\begin{equation} \label{ass:lfc} \tag{\sfemph{LFC}}
\cond{
\sfemph{locally finite}, i.e., for every vertex $x\in\Xs$ the cardinality 
$\deg(x)$ of its set of neighbours $\Nc(x)$ is finite, and 
\sfemph{connected}.
}
\end{equation}
We denote by $\Ec(\Xs)\subset\Xs\times\Xs$ the \sfemph{edge set} of $\Xs$ and write $d(x,y)$ for the \sfemph{graph distance} on~$\Xs\,$. 
We recall that the transition operator $P$ is always assumed to satisfy condition \eqref{ass:tc}, i.e.,
to be transient and to have pairwise communicating states. Here is a list of different
basic \emph{geometric}
adaptedness conditions. The random walk $(\Xs,P)$ with the transition probabilities
$p(x,y)=p_x(y)$ is said to be

{\setlength{\leftmargini}{23pt}
\begin{itemize}
\item
\sfemph{simple}, if for any $x\in\Xs$ the transition measure $p_x$ is equidistributed of its set of neighbours $\Nc(x)$, i.e., $p(x,y) = 1/\deg(x)$ for $[x,y]\in E(\Xs)$, and $p(x,y)=0$, otherwise.
\item
\sfemph{nearest neighbour}, if
$p(x,y)>0$ only when $[x,y]\in E(\Xs)$.
\item
of \sfemph{bounded range}, if there is $R<\infty$ such that $p(x,y)>0$ only when $d(x,y)\le R$.
\item
\sfemph{uniformly irreducible}, if there are $N<\infty$ and $\ep>0$ such that for any pair
$(x,y)\in\Ec(\Xs)$ there is a time $n\le N$ with $p^{(n)}(x,y)\ge\ep$.
\end{itemize}}

One can also impose various \emph{tightness} or \emph{moment} conditions on the distributions of the
distances $d(x,y)$ with respect to the transition probabilities $p(x,y)$, e.g., the \emph{uniform first
moment} condition from \dfnref{dfn:tight} (see \textsc{Kaimanovich -- Woess}
\cite[Section~3]{Kaimanovich-Woess92} for a detailed discussion).

\medskip

We now consider \emph{algebraic} adaptedness.
A graph $\Xs$ is called \sfemph{vertex transitive} if the action of its
\sfemph{group of automorphisms} $\Aut(\Xs)$ on the vertex set acts  transitively on the
vertex set. This is the case for the \emph{Cayley graph} of any finitely generated group with respect
to a finite symmetric set of generators $S$ (i.e., $[x,y]$ is an edge if and only if $x^{-1}y\in S$).
There are also vertex transitive graphs which are \emph{not} Cayley graphs.

Given a transition operator $P$ on a state space $\Xs$ (not necessarily endowed with a graph structure), one can define the \sfemph{automorphism group} of the Markov chain $(\Xs,P)$ as
$$
\Aut(\Xs,P)
= \{ g \in \Perm(\Xs) : p(gx,gy)=p(x,y) \quad \forall\, x,y \in \Xs \} \;.
$$
where $\Perm(\Xs)$ denotes the group of all permutations (not necessarily finitely supported!) of $\Xs$.
A natural algebraic adaptedness condition in this situation is to require that
$\Aut(\Xs,P)$ (or a subgroup) act transitively on $\Xs$, or at least \sfemph{quasi-transitively},
%
which consists in requiring that the action of the corresponding automorphism group have finitely
many orbits. 
This condition is satisfied for so-called \emph{random walks with internal degrees of freedom} (also known under numerous other names, in particular, as \emph{semi-Markov, covering,} or \emph{coloured} chains), e.g., see \textsc{Kaimanovich -- Woess} \cite{Kaimanovich-Woess02} and the references therein.

\subsection{A zoo of compactifications} \label{sec:zoo}

We now display several ``geometric'' compactifications of a graph $\Xs$ satisfying conditions \eqref{ass:lfc} and discuss if and how \thmref{thm:kconv} applies.

\begin{ex}[end compactification]
This definition goes back to Freudenthal \cite{Freudenthal45} and Halin \cite{Halin64}.
We denote by $\Cc(F)$ the (finite!) collection of all infinite connected components of
the graph $\Xs\setminus F$ obtained from $\Xs$ by removing a finite set of edges $F \subset \Ec(\Xs)$.
The \sfemph{end compactification} $\ol\Xs_E = \Xs \cup \p_E\Xs$ is the unique (up to homomorphisms)
minimal compactification of $\Xs$ to which all the indicator functions $\1_C$ of connected components
$C\in\Cc(F)$ extend continuously. The \sfemph{space of ends} $\p_E\Xs$ is the projective limit of the
discrete spaces $\Cc(F)$ as $F\to\Xs$. There is also a more explicit graph-theoretical description.
\end{ex}

\begin{ex}[hyperbolic compactification]
A graph $\Xs$ is called \sfemph{hyperbolic}, if it is a \emph{Gromov-hyperbolic metric space} with respect to the standard graph metric.  We refrain from re-stating all features of hyperbolic spaces and groups.
See the original paper by \textsc{Gromov} \cite{Gromov87} (and its numerous renditions), or, in the context
of random walks on graphs and groups, \textsc{Woess} \cite[\S 22]{Woess00}. A hyperbolic graph $\Xs$ has a
\sfemph{hyperbolic compactification} $\ol\Xs_H$ with the \sfemph{hyperbolic boundary} $\p_H \Xs$.
\end{ex}

\begin{ex}[Floyd compactification]
Let $\fl: \Zp \to (0,\infty)$ be a summable function such that there is $0<c<1$ with
$$
c\,\fl(n) \le \fl(n+1) \le \fl(n) \qquad \forall\,n\in\Zp \;.
$$
We define the \sfemph{$\fl$-length} $\ell_\fl(\pi)$ of any finite path $\pi$ in $\Xs$ as the sum of the $\fl$-lengths
$$
\ell_\fl(e) = \min\{ \fl(|x|), \fl(|y|)\}
$$
of its edges $e=[x,y]$, where as usual $|x| = d(x,o)$ denotes the graph distance between $x$ and a fixed root vertex $o$. The \sfemph{Floyd distance} on $\Xs$ is the resulting path metric
$$
d_\fl(x,y) = \inf \{ \ell_\fl(\pi) : \pi \textrm{ is a finite path from $x$ to $y$ } \} \;.
$$
We denote by $\ol\Xs_\fl$ the completion of $\Xs$ with respect to this metric, with the resulting
\emph{Floyd boundary} $\p_\fl\Xs = \ol\Xs_\fl \setminus \Xs$. The space $\ol\Xs_\fl$ is compact and does not
depend on the choice of the root $o$ (the identity map on $\Xs$ extends to a homeomorphisms between the
compactifications corresponding to different roots), see \textsc{Floyd} \cite{Floyd80} and \textsc{Karlsson}
\cite{Karlsson03, Karlsson03a, Karlsson03b}.
\end{ex}

The end, the hyperbolic (provided the graph is hyperbolic), and the Floyd compactifications have the following common features (see the aforementioned references):
{\setlength{\leftmargini}{23pt}
\begin{itemize}
\item
The action of the group of automorphisms $\Aut(\Xs)$ on $\Xs$ extends to its action on the whole compactification by homeomorphisms;
\item
If the graph is vertex-transitive, then the boundary of the compactification consists of one, two, or uncountably many points;
\item
All these compactifications are \emph{contractive $\Aut(\Xs)$-compactifications} in the sense of \textsc{Woess} \cite{Woess93}.
\end{itemize}}

The Floyd compactification is \sfemph{finer} than the end one, i.e., there exists a (necessarily surjective)
continuous map $\pi:\ol\Xs_\fl\to\ol\Xs_E$ (an extension of the identity map on $\Xs$) such that the embedding $\Xs\ha\ol\Xs_E$ is the composition of the embedding $\Xs\ha\ol\Xs_\fl$ with $\pi$. If the
graph $\Xs$ is hyperbolic, then the hyperbolic compactification is intermediate between the other two, i.e.,
it is coarser than the Floyd one and finer than the end one (the latter fact was, to our knowledge, first
explicitly stated by \textsc{Pavone} \cite{Pavone89}):
$$
\begin{tikzcd}
\ol\Xs_\fl \arrow[rd,dashrightarrow] \arrow[rr] & & \ol\Xs_E  \\
& \ol\Xs_H \arrow[ru,dashrightarrow] &
\end{tikzcd}
$$
Note that in general even the vertex-transitive graphs with infinitely many ends may be quite far from being
hyperbolic.

The following result from \textsc{Woess} \cite[Section 4]{Woess93} provides sufficient conditions for the
applicability of \thmref{thm:kconv}:

\begin{prp}
Let $\Xs$ be a 
graph satisfying  \eqref{ass:lfc}, and let $\ol\Xs$ be one of its
compactifications\,---\,the end, the hyperbolic (if $\Xs$ is hyperbolic), or the Floyd one\,---\,with
infinite boundary $\p\Xs$. If the group $\Aut(\Xs,P)$ acts quasi-transitively on $\Xs$ and 
does not fix a boundary
point, then the compactification is Dirichlet regular.
\end{prp}
\note{V: I removed the theorem on convergence in the end compactification - as far as I understand, its advantage over \prpref{prp:rd} is just in dealing with the case when the Green function does not vanish.
-- W: EXACTLY, but it also means that one has a case where the empirical distributions converge without Dirichlet regularity.}
The case when $\Aut(\Xs,P)$ fixes a boundary point is a very special one; we omit the details here.

Next, we review the situation when no group invariance is assumed. We recall that the \sfemph{spectral radius} of a transition operator $P$ is defined as
\begin{equation} \label{eq:rr}
r(P) = \limsup_{n\to\infty} \bigl[ p^{(n)} (x,y) \bigr]^{1/n} \,;
\end{equation}
under condition \eqref{ass:tc} it does not depend on the choice of $x,y\in\Xs$ (see, e.g. \textsc{Woess} \cite{Woess00}).

\begin{prp}\label{prp:rd} Suppose that \eqref{ass:tc} holds.
The end compactification is stochastically resolutive if one of the following two conditions holds:
\begin{itemize}
\item[{\rm (i)}]
$P$ has bounded range;
\item[{\rm (ii)}]
$P$ is uniformly irreducible, has a uniform first moment, and $r(P) < 1$.
\end{itemize}
Moreover, it is Dirichlet regular if, in addition to {\rm (i)} or {\rm (ii)}, the following respective conditions hold:
\begin{enumerate}
\item[{\rm (i$'$)}]
under condition (i): the Green kernel vanishes at infinity;
\item[{\rm (ii$'$)}]
under condition (ii): there are $C > 0$ and $r<1$ such that
\begin{equation}\label{eq:Crho}
p^{(n)}(x,y) \le C\,r^n \qquad \forall\,x,y \in \Xs,\; n \in \NN\,,
\end{equation}
\end{enumerate}
\end{prp}

\begin{rem} \label{rem:hypfl}
With small modifications, Proposition \ref{prp:rd} also holds for the other two compactifications.
 \\[3pt]
(a)  If the graph $\Xs$ is hyperbolic and $r(P) < 1$, then \prpref{prp:rd} holds for the hyperbolic compactification of $\Xs$ as well.
\\[3pt]
(b) For the Floyd compactifications in absence of a transitive group action, the case of simple random walk
has been touched  by \textsc{Karlsson} \cite{Karlsson03b}.
This has recently been generalised by \textsc{Spanos}~\cite{Spanos21p}; in particular, stochastic
resolutivity, resp. Dirichlet regularity hold under conditions (ii), resp. (ii)+(ii$'$).
\\[3pt]
For proofs and references regarding the end and hyperbolic cases, as well as
the question of validity of condition \eqref{eq:Crho}, see \textsc{Woess} \cite{Woess00}, in particular Sections 21--22.
\end{rem}

Boundary convergence of Markov chains is a vast and active area, and the purpose of the examples above is just to convey its flavour as a backdrop for \thmref{thm:kconv} rather than to provide any comprehensive overview. Without going into further details, let us mention, for instance, the related work on the visual compactification of Cartan -- Hadamard manifolds, the Busemann (or horospheric) compactification of metric spaces, various compactifications of Riemannian symmetric spaces, boundaries of planar graphs, the Thurston compactification of Teichm\"uller space, etc.\ etc.

Let us finally discuss a compactification of the state space $\Xs$ intrinsically determined just by the transition operator $P$. The latter, as always, is assumed to satisfy condition \eqref{ass:tc}, and therefore a normalisation of the Green kernel produces the \sfemph{Martin kernel}
$$
K_o(x,y) = \frac{G(x,y)}{G(o,y)} \;, \qquad x, y \in \Xs \;,
$$
where $o\in\Xs$ is a fixed reference point. The \sfemph{Martin compactification} is the unique (up to homeomorphisms) minimal compactification $\ol \Xs_M=\Xs\cup\p_M\Xs$ of the state space $\Xs$ to which each function $K_o(x,\cdot),\;x \in \Xs$, extends continuously in the second variable (e.g., see Woess \cite{Woess09} for a detailed exposition). 
The Martin compactification of a \emph{bounded range} Markov operator $P$ on a graph $\Xs$ is known to be comparable with the aforementioned geometric compactifications of $\Xs$ in the following situations:
\begin{itemize}
\item[{\rm (i)}]
it is finer than the end compactification;
\item[{\rm (ii)}]
it coincides with the hyperbolic compactification\,---\,if $\Xs$ is hyperbolic and $r(P)<1$;
\item[{\rm (iii)}]
it is finer than the Floyd compactification\,---\,if $(\Xs,P)$ is a random walk on (the Cayley graph of) a finitely generated group.
\end{itemize}
For (i) and (ii), see \textsc{Woess} \cite[Chapter IV]{Woess00} and the references therein; (iii) is very
recent and due to {\sc Gekhtman, Gerasimov, Potyagailo and Yang}
\cite{Gekhtman-Gerasimov-Potyagailo-Yang21}.

Every positive harmonic function $h$ has an integral representation
$$
h(x) = \int_{\p_M \Xs} K_o(x,\xi)\,d\nu_o^h(\xi)
$$
for a unique Borel measure $\nu_o^h$ on $\p_M\Xs$. The Martin compactification is stochastically
resolutive, and the hitting distribution $\nu_o$ issued from the reference point $o$ is precisely the
representing measure $\nu_o^\1$ of the constant harmonic function $\1(x)\equiv 1$. There are various classes
of Markov chains for which the Martin compactification is Dirichlet regular, but there are also many classes
for which it is not. In any case, at least claim (iii) of \thmref{thm:kconv} always applies to the Martin
compactification.

After mentioning that the Martin boundary considered as a measure space endowed with the family of the representing measures $\nu_o^\1$ is isomorphic to the \emph{Poisson boundary} of the chain, we shall now switch to discussing the boundary behaviour of branching Markov chains in the measure-theoretic setting.

\section{Boundary correspondence} \label{sec:corres}

\subsection{Motivation: topological case}

In the topological setup, as we saw in the previous Section (\thmref{thm:kconv} and \thmref{thm:emp}),
under suitable conditions there is a natural map
$
\;\Mb \mapsto \kb_\Mb\,,
$
which assigns to almost every sample path  of the branching Markov chain $\Mb=(M_n)$ a
finite positive weak* limit measure
$\;
\kb_\Mb \;
$
on a stationary space $\Ks$. The total mass
$\;
\| \kb_\Mb \| 
\;$
is the limit of the population martingale \eqref{eq:pm}, and its normalisation
$\;
\dol\kb_\Mb = \frac{1}{\| \kb_\Mb \|}\,\kb_\Mb 
\;$
is a random probability measure on $\Ks$ which can be interpreted as a limit of the population averages.
In particular, if $\Ks=\p\Xs$ is the boundary of a Dirichlet regular compactification
of the state space $\Xs$, then $\dol\kb_\Mb$ is the weak* limit of the empirical distributions
$$
\dol M_n = \frac{1}{\| M_n\|}\,M_n
$$
on the populations $M_n\,$.

The purpose of this section is to show that the limit measures associated with the sample paths of the branching Markov chain can also be defined by entirely measure-theoretical means not involving any topological convergence, as the transition probabilities of a certain Markov transfer operator acting between two measure spaces. This will allow us to define the limit distributions of the branching Markov chain on the measure-theoretical boundaries of the underlying chain.

Before proceeding further, let us return to \thmref{thm:emp}. It provides a measurable family of probability
measures $\dol\kb_\Mb$ on the stationary space $\Ks$ parameterised by the sample paths of the branching Markov chain $\Mb$. Considered as a Markov kernel from the path space $\Ms^\Zp$ to $\Ks$, this family gives rise to a positive norm 1 linear operator
\begin{equation} \label{eq:bc}
\Bs_\Ks: \ph \mapsto \Bs_{\raisemath{-2pt}{\Ks}} \ph \;, \qquad \text{where} \quad \ph\in C(\Ks),\quad \Bs_{\raisemath{-2pt}{\Ks}} \ph
= \langle\, \dol\kb_\Mb, \ph \rangle \;,
\end{equation}
from $C(\Ks)$ to the space $L^\infty(\Ms^\Zp)$ of bounded measurable functions on the path space of the branching chain with respect to the default measure class \eqref{fs}.

We will now go in the opposite direction by first defining an appropriate transfer operator and then using it to produce the associated family of boundary measures.

\subsection{Tail and Poisson boundaries} \label{sec:tp}

We begin with reminding the basic definitions and facts from the measurable boundary theory of Markov chains, see \textsc{Kaimanovich} \cite{Kaimanovich92} and the references therein for more details. Given a transition operator $P$ on a countable state space $\Xs$ (or, equivalently, the corresponding family of transition probabilities), this theory provides an integral representation of bounded harmonic functions, or, more generally, of bounded \sfemph{harmonic sequences} ($\equiv$ \sfemph{space-time harmonic functions})
$$
f_n = Pf_{n+1} \qquad \forall\,n\in\Zp \;.
$$
By $\Af_n^\infty$ we denote the $\si$-algebra on the path space $\Xs^\Zp$ determined by the positions of the chain at times~$\ge n$. The intersection
$$
\Af^\infty = \bigcap_n \Af_n^\infty
$$
is the \sfemph{tail $\si$-algebra} of the Markov chain $(\Xs,P)$, and it gives rise to the \sfemph{tail boundary}~$\Tc_P\Xs$ defined in the 
measure category 
by using \textsc{Rokhlin}'s correspondence between (complete) sub-$\si$-algebras of Lebesgue measure spaces and their quotient spaces; see e.g.
\textsc{Coud\`ene}~\cite[Chapter 15]{coudene16}. 
We denote the corresponding quotient map by
\begin{equation} \label{eq:tail}
\tail=\tail_P: \Xs^\Zp \to \Tc_P\Xs \;.
\end{equation}
The tail boundary is endowed with the \sfemph{harmonic measure class}, which is the $\tail$ image of the
default measure class \eqref{fs} on the path space, and the notation $L^\infty(\Tc_P\Xs)$ refers to the
harmonic measure class. Any initial position $(n,x)$ from the \sfemph{space-time} $\Zp\times X$ determines
the associated \sfemph{harmonic} probability measure $\et_{(n,x)}$ on~$\Tc_P^\Xs$ absolutely continuous with
respect to the harmonic measure class, and
\begin{equation} \label{eq:tailp}
f(n,x) = f_n(x) = \bigl\langle \et_{(n,x)}, \wh f\; \bigr\rangle
\end{equation}
is a space-time $P$-harmonic function for any $\wh f\in L^\infty(\Tc_P\Xs)$. Equivalently, $f$ is a
harmonic function of the \sfemph{space-time Markov chain} on $\Zp\times X$ (for which the spatial
transitions are accompanied by increasing the time coordinate by one). Conversely, a space-time 
function $f=(f_n)$ is $P$-harmonic if and only if the sequence of its values $f_n(X_n)$ along 
the sample paths of the Markov chain is a \emph{martingale} with respect to the increasing 
coordinate filtration of the path space. In particular, if $f$ is bounded, then the limit
\begin{equation} \label{eq:tailm}
\lim_n f_n(X_n) = \wh f (\tail\Xb)
\end{equation}
exists and is measurable with respect to the tail $\si$-algebra, i.e., it determines a function~$\wh f$
in $L^\infty(\Tc_P\Xs)$. Formulas \eqref{eq:tailp} and \eqref{eq:tailm} establish an isometric
isomorphism of $L^\infty(\Tc_P\Xs)$ and of the space of bounded space-time $P$-harmonic functions
endowed with the $\sup$ norm.

In the same way one defines (also in the measure category) the \sfemph{Poisson boundary} $\p_P X$
responsible for an integral representation of bounded $P$-harmonic functions (whence the name
alluding to the classical Poisson formula for harmonic functions on the unit disk). It is the
quotient of the path space under the \sfemph{boundary map}
$$
\bnd=\bnd_P:\Xs^\Zp\to\p_P X \;,
$$
determined by the \sfemph{exit $\si$-algebra} (the sub-algebra of the tail $\si$-algebra consisting of shift invariant sets\,---\,this is why this $\si$-algebra is also sometimes called \emph{invariant}). The following commutative diagram illustrates the relationship between the path space, the tail and the Poisson boundaries:
\begin{equation} \label{eq:cd}
\begin{tikzcd}
\Xs^\Zp \arrow[rd,"\bnd"'] \arrow[r,"\tail"] & \Tc_P\Xs \arrow[d,"\pe"] \\
& \p_P\Xs
\end{tikzcd}
\end{equation}
The Poisson boundary can be interpreted as the \emph{space of ergodic components} of the transformation $T$ of the tail boundary induced by the time shift on the path space, and the resulting projection is the map $\pe:\Tc_P\Xs\to\p_P\Xs$ in the above diagram. Formulas \eqref{eq:tailp} and \eqref{eq:tailm} restricted to the space of bounded $P$-harmonic functions (i.e., of space-time harmonic functions constant in time) establish its isometric isomorphism with the subspace of $L^\infty(\Tc_P\Xs)$ that consists of $T$-invariant functions. In terms of the Poisson boundary this isomorphism takes the form of the \sfemph{Poisson formula}
$$
f(x) = \bigl\langle \nu_x, \wh f\; \bigr\rangle \;,
$$
where $\wh f\in L^\infty(\p_P X)$, and $\nu_x$ are the \sfemph{harmonic measures} on the Poisson boundary, i.e., the images of the measures $\Pb_{\!x}$ on the path space under the boundary map $\bnd$\,---\, or, equivalently, the images of the measures $\et_x=\et_{(0,x)}$ on $\Tc_P\Xs$ under the quotient map $\pe:\Tc_P\Xs\to\p_P\Xs\,$.

We now provide a  characterisation of the topological $P$-boundaries of a Markov chain in terms of 
quotients of the Poisson boundary 
mentioned after \dfnref{dfn:pb}.

\begin{prp} \label{prp:bdry}
Let $P$ be a Markov operator on a countable state space $\Xs$, and let $(\Ks,\ka)$ be a compact separable $P$-stationary space. It is a $P$-boundary if and only if, as a measure space, it is a quotient of the Poisson boundary $\p_P\Xs$, i.e., there exists a measurable map $\qe:\p_P\Xs\to\Ks$ such that $\qe(\nu_x)=\ka_x$ for all $x\in\Xs$.
\end{prp}

\begin{proof}
Let $\Ks$ be a $P$-boundary. Then for almost every sample path of the Markov chain $\Xb=(X_n)$
the weak* limit
$$
\kb_\Xb = \wlim_{t\to\infty} \ka_{X_t}
$$
is a delta measure, which provides a map from the path space to $\Ks$ which is measurable with respect
to the exit $\si$-algebra, i.e., a sought for measurable map $\qe:\p_P\Xs\to\Ks$, and by formula \eqref{eq:ikbxb} from \prpref{prp:c} $\qe(\nu_x)=\ka_x$ for all $x\in\Xs$.

Conversely, let $\Ks$, as a measure space, be a quotient of the Poisson boundary. Then any test function $\ph\in C(\Ks)$, considered as an element of $L^\infty(\Ks)$, can be lifted to a function $\wh f\in L^\infty(\p_P\Xs)$. Let $f=f^\ph$ be the associated bounded harmonic function:
$$
f(x) = \langle \nu_x, \wh f \rangle = \langle \ka_x, \ph \rangle \;, \qquad x\in\Xs \;.
$$
Then for almost every sample path of the Markov chain $\Xb=(X_n)$,
$$
\langle \ka_{X_n}, \ph \rangle = f(X_n) \to \wh f(\bnd \Xb) = \ph(\qe\circ\bnd \Xb) \;,
$$
i.e.,
$$
\kb_\Xb = \de_{\qe\circ\bnd \Xb} \;.
$$
\end{proof}

\subsection{Boundaries of branching Markov chains} \label{sec:bbmc}

Now we pass to the branching chain on $\Ms$ determined by the transition probabilities $\Pi_m$ \eqref{eq:PiM}, or, equivalently, by the transition operator $\Pc$ \eqref{eq:pe}.

We denote the tail boundary of the branching Markov chain by $\Tc_\Pc\Ms$, and the Poisson boundary by $\p_\Pc\Ms$. By $\upet_{(t,m)}$ and $\upnu_m$ we denote the harmonic measures on the respective tail and Poisson boundaries corresponding to an initial population $m\in\Ms$ (replacing, as in \S \ref{subsec:bd}, the subscript $\de_x$ with $x$ if the initial population is a singleton $\de_x$).

We recall that in what concerns \emph{random walks on countable groups}
(cf.\ \exref{ex:gr}), the difference between the tail and the Poisson boundaries is not very significant. They do coincide in the aperiodic case (\textsc{Derriennic} \cite{Derriennic76}); otherwise the fibres of the projection $\pe$ \eqref{eq:cd} of the tail boundary onto the Poisson boundary are parameterised by the periodicity classes of the random walk (\textsc{Jaworski} \cite{Jaworski95}), in particular, $\pe$ is always a bijection with respect to a one-point initial distribution (the latter fact plays a key role in the entropy theory of random walks on groups, see \textsc{Derriennic} \cite{Derriennic80} and \textsc{Kaimanovich~-- Vershik} \cite{Kaimanovich-Vershik83}). The situation is similar for \emph{random walks on graphs} under the \sfemph{uniform ellipticity} condition (the transition probabilities between any two neighbouring vertices are bounded away from 0): in this case the tail and the Poisson boundaries also coincide with respect to any one-point initial distribution, and the cardinality of the fibres of the projection $\pe$ is at most 2, see \textsc{Kaimanovich} \cite[Corollary 2 on p. 162]{Kaimanovich92}.

Branching Markov chains are manifestly space inhomogeneous, and the difference between their tail and
Poisson boundaries is much more pronounced. It can be illustrated already by the simplest example:

\begin{ex}\label{ex:GW-P}
Consider the
Galton -- Watson process (\exref{ex:GW}). Let $\rho>1$ be the mean of a non-degenerate offspring
distribution $\pi\in\Meas(\NN)$ that satisfies the $L\log L$ moment condition. Then, assuming the non-
extinction condition \eqref{ass:ne} (as always in this paper), for almost every sample path  of the Galton
-- Watson process $(Z_n)$ there exists the limit
$$
W_\infty=\lim_n Z_n/\rho^{n} > 0
$$
(cf.\ \secref{subsec:mart} and \secref{subsec:hm}) which is clearly tail measurable. It was proved by
\textsc{Lootgieter} \cite[Corollaire 2.3.II, Corollaire 3.3.II]{Lootgieter77} that the limit $W_\infty$
completely describes the tail behaviour of the Galton -- Watson process with respect to any one-point
initial distribution, or, in the aperiodic case, with respect to any initial distribution. 
(According to  \textsc{Cohn} \cite[pp.\ 420-421]{Cohn79}, this result was also independently obtained by
B.\ M.\ Brown in an unpublished 1977 manuscript ``The tail $\si$-field of a branching process''.)
Therefore, under  the $L\log L$ moment condition the tail boundary of the Galton~-- Watson process
coincides with the product  of the positive ray $\RR_+$ by the finite set of periodicity classes, or just with $\RR_+$ in the aperiodic
case. The arising limit measure on $\RR_+$ (the distribution of $W_\infty$) is actually absolutely
continuous, and its support is the whole ray~$\RR_+$, see \textsc{Athreya~-- Ney} \cite[Theorem I.10.4]
{Athreya-Ney72}. The time shift on the path space amounts to the multiplication of the limit $W_\infty$ by $
\rho$. The corresponding action of the group $\ZZ$ on~$\RR_+$ is dissipative, and therefore its space of
ergodic components ($\equiv$ the Poisson boundary of the Galton -- Watson process) can be identified with
the fundamental interval~$[1,\rho)\,$.

This identification of the Poisson boundary of Galton -- Watson processes was first obtained\,---\,in somewhat different terms\,---\,by \textsc{Dubuc} \cite[Theorem 2]{Dubuc71} under the finite second moment condition.
\end{ex}

\begin{rem}
It seems plausible that the tail and the Poisson boundaries admit a similar description for branching Markov chains over any \emph{finite} state space ($\equiv$ multi-type branching processes with a finite number of types). As far as we know, this question has not been addressed in the literature.
\end{rem}

If one passes to branching Markov chains over an infinite state space, then the problem of identification
of the tail and the Poisson boundaries appears to be horizon-less. This is indicated by the abundance
of various martingales already in the simplest case of branching random walks on $\ZZ$ (e.g., see
\textsc{Shi} \cite[Chapter~3]{Shi15}). To the best of our knowledge, this problem has not been
formulated even in the aforementioned $\ZZ$ case. We are now going
to provide links between measure-theoretic boundaries of a branching
Markov chain and that of the underlying chain.

\subsection{Harmonic martingales} \label{subsec:hm}

Given two functions $f$ and $g$ on the state space $\Xs$, their respective lifts $\wt f$ and $\wt g$ to $\Ms$ (see \secref{subsec:trans}) have the property that $\Pc\wt f=\wt g$ if and only if $\ol\pi\cdot Pf=g$,
see \prpref{prp:tr} and recall that $x \mapsto \ol\pi_x$ is the assignment of the branching ratio at
$x$. We immediately get the following.

\begin{prp} \label{prp:hlift}
The lifts $\wt f_n$ of a sequence of functions $f_n$ on $\Xs$ form a space-time $\Pc$-harmonic function
on $\Ms$, i.e.,
$$
\wt{f_n} =\Pc\,\wt f_{n+1}  \qquad\forall\,n\in\Zp
$$
if and only if on $\Xs$
\begin{equation} \label{eq:fn}
f_n = \ol\pi\cdot P f_{n+1} \qquad\forall\,n\in\Zp \;.
\end{equation}
\end{prp}

\begin{cor}\label{cor:st}
If condition \eqref{ass:br} is satisfied, then for any space-time $P$-harmonic function $(f_n)$ on
the state space $\Xs$ the sequence
$$
\left( \frac{1}{\rho^n}\, \wt f_n \right)
$$
is a space-time $\Pc$-harmonic function on the population space $\Ms$.
\end{cor}

\begin{dfn} \label{dfn:hm}
The \sfemph{harmonic martingale} determined by a space-time $P$-harmonic function $f=(f_n)$ on $\Xs$
is the sequence of random variables
\begin{equation} \label{eq:hm}
W^f_n (\Mb) = \frac{\wt f_n(M_n)}{\rho^n} = \frac{\left\langle M_n, f_n \right\rangle}{\rho^n}
\end{equation}
on the path space of the branching Markov chain $\Mb=(M_n)$.
\end{dfn}

We denote the pointwise (almost everywhere) limit of the harmonic martingale \eqref{eq:hm} by
$$
W^f_\infty(\Mb) = \lim_n W^f_n(\Mb) \;.
$$
The random variable $W^f_\infty$ is tail measurable, and therefore it can be presented as the composition
\begin{equation} \label{eq:wft}
W^f_\infty = \we^f \circ \tail_\Pc
\end{equation}
of the quotient map $\tail_\Pc: \Ms^\Zp \to \Tc_\Pc\Xs$ \eqref{eq:tail} with the arising measurable
function $\we^f$ on the tail boundary $\Tc_{\Pc}\Ms$ of the branching Markov chain.

In the particular case when $f=\1$ the associated harmonic martingale $(W_n^\1)$ is precisely the
\emph{population martingale} $(W_n)$ introduced in \dfnref{dfn:pop} and studied in \secref{sec:pm}. We denote by $\we=\we^\1$ the function on the tail boundary determined by the pointwise limit $W_\infty=W_\infty^\1$ of the population martingale (the \emph{limit population ratio}, see \dfnref{dfn:pop}).

If the limit population ratio is positive (by \thmref{thm:pos} this is almost surely the case under the uniform $L\log L$ moment condition), then
\begin{equation} \label{eq:ratio}
\begin{aligned}
\frac{\we^f}{\we}(\tail\Mb)
&= \frac{W^f_\infty}{W_\infty}(\Mb)
= \lim_n \frac{W^f_n}{W_n}(\Mb) \\
&= \lim_n \frac{\left\langle M_n\,, f_n \right\rangle}{\|M_n\|}
= \lim_n \Bigl\langle \dol M_n\,, f_n \Bigr\rangle
\end{aligned}
\end{equation}
is nothing but the \emph{limit empirical average} of the functions $f_n$ along 
the branching Markov chain.

\begin{rem} \label{rem:hm}
In the context of branching Markov chains the term ``harmonic martingale'' was used by
\textsc{Biggins -- Cohn -- Nerman} \cite{Biggins-Cohn-Nerman99} for the sequence (in our notation)
$\wt\ph_n(M_n)$, where $(\ph_n)$ is a sequence of functions on $\Xs$ such that its lift to the space of populations $(\wt\ph_n)$ is a space-time $\Pc$-harmonic function. Actually, in the setup of
\cite{Biggins-Cohn-Nerman99} the state space $\Xs$ is endowed with a \sfemph{space-time partition}
into pairwise disjoint \sfemph{levels} \mbox{$\Xs_n,\;n\ge 0$} such that $\Xs_0$ consists of a single
state~$x_0$, the branching chain starts at time 0 from a single particle sitting at $x_0$, and at each
step the population moves to the next level, so that the time $n$ random population $M_n$ is concentrated
on $\Xs_n$. Therefore, the sequence~$(\ph_n)$ can be considered as a \emph{single} function $\ph$ on the
state space $\Xs$ with the property that its lift $\wt\ph$ is a $\Pc$-harmonic function on the population
space $\Ms$. Such functions on $\Xs$ are called \emph{mean-harmonic} by \textsc{Biggins -- Kyprianou}
\cite{Biggins-Kyprianou04}. Our setup is slightly more general as we deal with the space-time harmonic
functions which do not necessarily come from a space-time partition of the state space. We feel that it is
more consistent to deal with the space-time harmonic functions (instead of the time constant ones) from the
very beginning. The reason is that martingales, by their very nature, are linked with the \emph{tail}
$\si$-algebra (rather than with the exit one), cf.\ \secref{sec:tp} and \thmref{thm:bdry}.
\end{rem}

\subsection{Boundary transfer operator} \label{subsec:transfer}

Below we are going to use the standard facts from the measurable theory of Markov operators,
e.g., see \textsc{Foguel} \cite{Foguel80}. We recall that, given two measure spaces $(X,\mu_X)$ and $(Y,\mu_Y)$, a linear operator
$$
B:L^\infty(Y,\mu_Y)\to L^\infty(X,\mu_X)
$$
is called \sfemph{Markov} if it preserves constants, is positive and order continuous, i.e.,
\begin{enumerate}[{\rm (i)}]
\item
$B \1_Y = \1_X$\;,
\item
$B\ph\ge 0$ for any $\ph\ge 0$ \;,
\item
$B \ph_k \da 0$ for any sequence $\ph_k\da 0$ \;.
\end{enumerate}

As we have explained in \secref{sec:tp}, formulas \eqref{eq:tailp} and \eqref{eq:tailm} establish an isometric isomorphism $\wh f \mapsto f=(f_n)$ of $L^\infty(\Tc_P\Xs)$ and of the Banach space of bounded space-time $P$-harmonic functions on the state space $\Xs$ endowed with the $\sup$ norm.

\begin{thm} \label{thm:bdry}
If the offspring distributions of a branching Markov chain with property \eqref{ass:ne}
satisfy the uniform $L\log L$ moment condition, then the map
\begin{equation} \label{eq:b}
B: \wh f \mapsto \frac{\we^f}{\we} \;, \qquad L^\infty(\Tc_P\Xs) \to L^\infty(\Tc_\Pc\Ms) \;,
\end{equation}
is a Markov operator. Here $\we^f/\we$ is the function \eqref{eq:ratio} on the tail boundary $\Tc_\Pc\Ms$ that represents the limit empirical averages of the space-time
harmonic function $f=(f_n)$ determined by $\wh f$.
\end{thm}

\begin{proof}
Property (i) from the definition of a Markov operator follows from \thmref{thm:pos}, whereas (ii) is obvious. We just have to verify property (iii), i.e., the order continuity of~$B$. It is here that we use the uniform integrability of the population martingale which implies that the operator $B$ preserves the integrals with respect to appropriately chosen measures on the tail boundaries $\Tc_P\Xs$ and $\Tc_\Pc\Ms$.

The first observation is that the uniform integrability of the population martingale (\thmref{thm:ll}) implies the uniform integrability of the harmonic martingale $\bigl(W^f_n\bigr)$ for any \emph{bounded} space-time harmonic function $f=(f_n)$ on $\Xs$. Thus, for any initial population $m\in\Ms$
$$
\langle \, \pmb\Pe_m, W^f_0 \rangle = \langle \, \pmb\Pe_m, W^f_\infty \rangle \;,
$$
or, in view of \eqref{eq:hm} and \eqref{eq:wft},
$$
\langle m, f_0 \rangle = \langle \upet_m, \we^f \rangle \;,
$$
where $\upet_m=\tail(\pmb\Pe_m)$ is the harmonic measure on the tail boundary $\Tc_\Pc\Ms$ corresponding to the initial distribution $\de_m$. In terms of the boundary function $\wh f\in L^\infty(\Tc_P\Xs)$ representing~$f$ we then have
\begin{equation} \label{eq:epep}
\bigl\langle \et_m, \wh f \,\bigr\rangle = \bigl\langle \we\!\cdot\!\upet_m, B\wh f\, \bigr\rangle \;,
\end{equation}
where
$$
\et_m = \sum_{x\in m} \et_x \;,
$$
and $\we\!\cdot\!\upet_m$ is the measure on $\Tc_\Pc\Ms$ with the density $\we$ with respect to the harmonic measure~$\upet_m\,$, so that
$$
\|\et_m \| = \| \we\!\cdot\!\upet_x \| = \|m\| \;.
$$

In view of the monotone convergence theorem, identity \eqref{eq:epep} then implies that if
${\wh f}^{\,(k)}\in L^\infty(\Tc_P\Xs)$ with ${\wh f}^{\,(k)}\da 0$ almost everywhere, then
$B{\wh f}^{\,(k)}\da 0$ almost everywhere with respect to all measures $\we\!\cdot\!\upet_m\,,\;m\in\Ms$, which by \thmref{thm:pos} is the same as the almost everywhere convergence with respect to the harmonic measure class on $\Tc_\Pc\Ms$.
\end{proof}

If a measurable space $\Ks$ is a quotient of the tail boundary $\Tc_P\Xs$, then the precomposition of the transfer operator $B$ \eqref{eq:b} constructed in \thmref{thm:bdry} with the lift
$$
L^\infty(\Ks)\to L^\infty(\Tc_P\Xs)
$$
provides a Markov operator
\begin{equation} \label{eq:bks}
B_\Ks:L^\infty(\Ks)\to L^\infty(\Tc_\Pc\Ms) \;,
\end{equation}
which we are going to compare with the operator $\Bs_\Ks$ \eqref{eq:bc}. Since the map $\Mb\mapsto\kb_\Mb$ is tail measurable by the definition of the measures $\kb_\Mb$ in \thmref{thm:kconv}, the operator~$\Bs_\Ks$ produces tail measurable functions on the path space, and therefore its range can be identified with the space $L^\infty(\Tc_\Pc\Ms)$.

\begin{thm} \label{thm:bb}
Under the conditions of \thmref{thm:bdry}, if $\Ks$ be a compact separable $P$-boundary of the underlying Markov chain $(\Xs,P)$, then the restriction of the operator $B_\Ks$~\eqref{eq:bks} to the space $C(\Ks)$ coincides with the operator $\Bs_\Ks$ \eqref{eq:bc}.
\end{thm}

\begin{proof}
Take a test function $\ph\in C(\Ks)$. Then by the definition of the operator $\Bs_\Ks$ and by \thmref{thm:emp} for almost every sample path $\Mb=(M_t)$
$$
\Bs_{\raisemath{-2pt}{\Ks}} \ph (\tail\Mb) = \langle \dol\kb_\Mb, \ph \rangle = \lim_t \langle \dol\ka_{M_t}, \ph \rangle \;,
$$
whereas by \eqref{eq:ratio}
$$
B_{\raisemath{-2pt}{\Ks}} \ph (\tail\Mb) = \lim_t \Bigl\langle \dol M_t, f \Bigr\rangle \;,
$$
where $f(x)=\langle \ka_x,\ph \rangle$, and therefore $\Bigl\langle \dol M_t, f \Bigr\rangle=\langle \dol\ka_{M_t}, \ph \rangle$.
\end{proof}

\subsection{Boundary measures}

\begin{thm}\label{thm:bmeas}
If the offspring distributions of a branching Markov chain satisfy the uniform $L\log L$ moment condition,
then the tail boundary of the underlying chain $\Tc_P\Xs$ is endowed with a family of probability
measures $\et^{\,\xi}$ indexed by the points $\xi\in\Tc_\Pc$ from the tail boundary of the branching
chain with the following properties:
\begin{itemize}
\item[{\rm (i)}]
The family $\{\et^{\,\xi}\}$ is measurable in the sense that for any function $\ph\in L^\infty(\Tc_P\Xs)$ the integrals $\langle \et^{\,\xi}, \ph \rangle$ depend on $\xi$ measurably.
\item[{\rm (ii)}]
If $\Ks$ is a compact separable $P$-boundary, then for almost every sample path of the branching chain
$\Mb=(M_n)$, the limit measure $\dol\kb_\Mb$ on $\Ks$ from \thmref{thm:emp} is the image
$\qe(\et^{\,\xi})$ of the measure $\et^{\,\xi},\;\xi=\tail\Mb,$ under the quotient map $\qe:\Tc_P\Xs\to\Ks$.
\item[{\rm (iii)}]
In particular, if $\p\Xs$ is the boundary of a Dirichlet regular compactification of the state space
$\Xs$, then the empirical distributions $\dol M_n$ almost surely weak* converge to the image of the
measure $\et^{\,\xi},\;\xi=\tail\Mb,$ under the quotient map from the tail boundary $\Tc_P\Xs$ to $\p\Xs$.
\end{itemize}
\end{thm}

\begin{proof}
We will construct the measures $\et^{\,\xi}$ as the transition probabilities of the Markov operator $B$ from \thmref{thm:bdry}.

We denote by $L^1(\Tc_P\Xs)$ the Banach space of finite measures absolutely continuous with respect to the
harmonic measure class on the tail boundary $\Tc_P\Xs$ of the underlying chain and endowed with the total
variation norm. We emphasise that this space\,---\,in the same way as its dual $L^\infty(\Tc_P\Xs)$\,---\,is
defined ``coordinate free'', just in terms of the harmonic \emph{measure class}. If one takes a reference
measure $\la$ from this class, then the elements of $L^1(\Tc_P\Xs)$ can be identified with their densities
with respect to $\la$, after which $L^1(\Tc_P\Xs)$ becomes the ``usual'' space $L^1(\Tc_P\Xs,\la)$.
Likewise, we denote by $L^1(\Tc_P\Ms)$ the analogous space associated with the tail boundary of the
branching Markov chain.

Being Markov, the operator
$$
B: L^\infty(\Tc_P\Xs) \to L^\infty(\Tc_\Pc\Ms)
$$
from \thmref{thm:bdry} is dual to an operator
$$
\la\mapsto \la B \;, \qquad L^1(\Tc_\Pc\Ms) \to L^1(\Tc_P\Xs) \;,
$$
and, as we saw in the course of the proof of \thmref{thm:bdry}, formula \eqref{eq:epep},
$$
\et_m = (\we\!\cdot\!\upet_m) B \qquad\forall\,m\in\Ms \;.
$$

For any initial probability measure $\la\in L^1(\Tc_\Pc\Ms)$ the operator $B$ gives rise to the associated
joint distribution $\la^B$ on the product $\Tc_\Pc\Ms\times\Tc_\Pc\Xs$ whose marginal distributions are the
measures $\la$ and $\la B$. Since all involved measure spaces are Lebesgue spaces, the conditional measures
$\et^{\,\xi},\; \xi\in \Tc_\Pc\Ms,$ on the fibres $\{\xi\}\times \Tc_\Pc\Xs\cong \Tc_\Pc\Xs$ of the
projection
$$
\Tc_\Pc\Ms\times\Tc_\Pc\Xs \to \Tc_\Pc\Ms
$$
are well-defined and their dependence on $\xi$ is measurable. Their system does not depend (mod 0) on the choice of $\la$ and provides the transition probabilities that determine the operator $B$, so that
$$
\la B = \int \et^{\,\xi}\,d\la(\xi)
$$
for any $\la\in L^1(\Tc_\Pc\Ms)$.

Claim (ii) then follows from \thmref{thm:bb}. Indeed, the operator $B_\Ks$ \eqref{eq:bks} being the result of the precomposition of the operator $B$ \eqref{eq:b} with the lift $L^\infty(\Ks)\to L^\infty(\Tc_P\Xs)$ determined by the quotient map $\qe:\Tc_P\Xs\to\Ks$, its transition probabilities are the $\qe$-images of the transition probabilities $\et^{\,\xi}$ of the operator $B$. On the other hand, by \thmref{thm:bb} the operator $B_\Ks$ has the same transition probabilities as the operator $\Bs_\Ks$ \eqref{eq:bc}, whereas the latter ones are, by definition, the measures $\dol\kb_\Mb\,$.

Finally, claim (iii) now follows from \thmref{thm:emp}.
\end{proof}

\begin{rem}
Although the measure $\la B$ on $\Tc_P\Xs$ is absolutely continuous with respect to the harmonic measure class, the measures $\et^{\,\xi}$ need \emph{not} be absolutely continuous. An extreme example is provided by the situation when there is no branching ($\rho=1$), and the branching Markov chain is reduced to an ordinary one (\exref{ex:gw}).
\end{rem}

\bibliographystyle{amsalpha}
\bibliography{brw}

\end{document}